\documentclass[a4paper,11pt]{amsart}

\newfont{\cyr}{wncyr10 scaled 1100}

\usepackage[usenames,dvipsnames]{color}

\usepackage[left=2.7cm,right=2.7cm,top=3.5cm,bottom=3cm]{geometry}
\usepackage{amsthm,amssymb,amsmath,amsfonts,mathrsfs,amscd,yfonts}
\usepackage{turnstile}
\usepackage[latin1]{inputenc}
\usepackage[all]{xy}
\usepackage{latexsym}
\usepackage{longtable}
\usepackage{hyperref}

\theoremstyle{plain}
\newtheorem{theorem}{Theorem}[section]

\newtheorem{lemma}[theorem]{Lemma}

\newtheorem{propo}[theorem]{Proposition}

\theoremstyle{definition}

\newtheorem{defi}[theorem]{Definition}

\newtheorem{examplewr}[theorem]{Example}

\theoremstyle{remark}
\newtheorem{obswr}[theorem]{Observation}
\newtheorem{remarkwr}[theorem]{Remark}

\newenvironment{remark}{\begin{remarkwr}\begin{upshape}}{\end{upshape}\end{remarkwr}}

\DeclareMathOperator{\dR}{\mathrm{dR}}

\DeclareMathOperator{\wt}{wt}

\DeclareMathOperator{\BK}{BK}

\DeclareMathOperator{\cyc}{cyc}
\DeclareMathOperator{\fin}{f}

\DeclareMathOperator{\Spf}{Spf}

\DeclareMathOperator{\triv}{triv}

\DeclareMathOperator{\ac}{ac}
\DeclareMathOperator{\loc}{loc}

\DeclareMathOperator{\cont}{cont}
\DeclareMathOperator{\ad}{ad}

\newcommand{\cC}{\mathcal C}

\newcommand{\Gal}{\mathrm{Gal\,}}

\newcommand{\ord}{{\mathrm{ord}}}
\newfont{\gotip}{eufb10 at 12pt}

\newcommand{\cL}{{\mathcal L}}

\newcommand{\hf}{{\mathbf{f}}}
\newcommand{\hg}{{\mathbf{g}}}
\newcommand{\hh}{{\mathbf h}}

\DeclareMathOperator{\Hom}{Hom}

\include{thebibliography}

\begin{document}

\title[The exceptional zero phenomenon for elliptic units]{The exceptional zero phenomenon for elliptic units}

\author{\'Oscar Rivero}

\begin{abstract}
The exceptional zero phenomenon has been widely studied in the realm of $p$-adic $L$-functions, where the starting point lies in the foundational works of Ferrero--Greenberg, Katz, Gross, or Mazur--Tate--Teitelbaum, among others. This phenomenon also appears in the study of Euler systems, and in this case one is led to study higher order derivatives of cohomology classes in order to extract the arithmetic information which is usually encoded in the explicit reciprocity laws which make the connection with $p$-adic $L$-funtions. In this work, we focus on the elliptic units of an imaginary quadratic field and study this exceptional zero phenomenon, proving an explicit formula relating the logarithm of a {\it derived} elliptic unit either to special values of Katz's two variable $p$-adic $L$-function or to its derivatives. Further, we interpret this fact in terms of an $\mathcal L$-invariant and relate this result to other approaches to the exceptional zero phenomenon, most notably to the work of Bley, and we also compare this setting with other scenarios concerning Heegner points and Beilinson--Flach elements.
\end{abstract}

\address{O. R.: Departament de Matem\`{a}tiques, Universitat Polit\`{e}cnica de Catalunya, C. Jordi Girona 1-3, 08034 Barcelona, Spain}
\email{oscar.rivero@upc.edu}

\subjclass[2010]{11G16; 11F67}

\maketitle

\tableofcontents

\section{Introduction}
Since the introduction of the exceptional zero phenomenon for the Kubota--Leopoldt $p$-adic $L$-function by Ferrero and Greenberg \cite{FG} and for the $p$-adic $L$-function attached to an elliptic curve by Mazur, Tate and Teitelbaum \cite{MTT}, a lot of progress has been made in the study of this topic. The main goal of this article is to study exceptional zero phenomena for Katz's two-variable $p$-adic $L$-function at points lying {\em outside} the region of classical interpolation, where the Euler system of elliptic units vanishes. Hence, our setting departs notably from loc.\,cit.\,and is closer in spirit to the study of e.\,g.\,\cite{Cas} and \cite{RiRo}, sharing some points in common with earlier work of Solomon \cite{Sol} and Bley \cite{Bl}, and also with the recent preprint of B\"uy\"ukboduk and Sakamoto \cite{BS}.

Fix once and for all a prime $p$ and a quadratic imaginary field $K$ in which $p$ splits, and fix embeddings $\mathbb C \hookleftarrow \bar{\mathbb Q} \hookrightarrow \mathbb C_p$. Let $\mathfrak p$ and $\bar{\mathfrak p}$ be the two prime ideals lying over $p$, with $\mathfrak p$ the prime above $p$ induced by the previously fixed embedding $\bar{\mathbb Q} \hookrightarrow \mathbb C_p$ (assume in the introduction for notational simplicity that both ideals are principal, and with a slight abuse of notation we denote by $\mathfrak p$ and $\bar{\mathfrak p}$ their generators). Let $K_{\infty}^{\cyc}$ and $K_{\infty}^{\ac}$ be the cyclotomic and anticyclotomic $\mathbb Z_p$-extensions of $K$, respectively, and set $K_{\infty} = K_{\infty}^{\ac} K_{\infty}^{\cyc}$. Denote $\Gamma_K = \Gal(K_{\infty}/K)$ and $\Lambda_K = \mathbb Z_p[[\Gamma_K]]$. The latter is a Galois module with an appropriate {\it tautological action} that we later recall. The weight space is the formal spectrum $\Spf(\Lambda_K)$ of the two-variable Iwasawa algebra $\Lambda_K$. Let $\psi$ stand for a Hecke character of finite order, conductor $\mathfrak n$ and taking values in a number field $F$, and let $\mathcal N$ denote the norm character of $K$. We denote by $N$ the norm of $\mathfrak n$, and assume that $(N,p)=1$. Finally, let $F_p$ denote a completion of $F$ at the prime $p$. The Hecke character $\psi$ may be also understood as a Galois character $\psi: G_K \rightarrow F^{\times}$; the notation $\psi'$ will be used to designate the composition of $\psi$ with conjugation by the non-trivial element in $\Gal(K/\mathbb Q)$: $\psi'(\sigma) = \psi(\tau \sigma \tau^{-1})$, where $\tau$ is any element of $G_{\mathbb Q}$ which acts non-trivially on $K$.

As an additional piece of notation, let $H_{\infty}$ denote the unique $\mathbb Z_p$-extension of $K$ in which $\bar{\mathfrak p}$ is unramified (therefore, the prime $\mathfrak p$ is ramified in $H_{\infty}/K$). We choose a Galois character $\lambda$ of $\Gamma_K$, so that it factors through $\Gal(H_{\infty}/K)$, defining an isomorphism \[ \Gal(H_{\infty}/K) \longrightarrow 1+p\mathbb Z_p. \] The choice of $\lambda$ is unique once we require that it is the Galois representation corresponding to a Gr\"ossencharacter for $K$ of infinity type $(1,0)$. We define in the same way an extension $H_{\infty}'$ and a character $\lambda'$, exchanging the roles of $\mathfrak p$ and $\bar{\mathfrak p}$. Although characters of $\Gamma_K$ are elements of $\Hom_{\cont}(G_K, \bar{\mathbb Q}_p^{\times})$, we are interested in the restriction to the subspace of $\Hom_{\cont}(G_K,\bar{\mathbb Q}_p^{\times})$ given by \[ \Sigma_{\psi} := \{ \psi \xi \mathcal N^s \lambda^t \text{ such that } (s,t) \in \mathbb Z_p^2 \}, \] where $\xi$ is a finite order character of $p$-power conductor.

One can naturally associate, via the theory of elliptic units of Robert \cite{Rob} (see also de Shalit and Yager's approaches \cite{deS},\,\cite{Yag}) and Kummer maps, a global cohomology class to $\psi$
\begin{equation}
\kappa_{\psi,\infty} \in H^1(K,\Lambda_K \otimes F_p(\psi^{-1})(\mathcal N)).
\end{equation}
We denote by $\kappa_{\psi} \in H^1(K,F_p(\psi^{-1}))$ and $\kappa_{\psi \mathcal N} \in H^1(K,F_p(\psi^{-1})(\mathcal N))$ the specializations of $\kappa_{\psi,\infty}$ at $\psi$ and $\psi \mathcal N$, respectively.
In section 4, we prove that $\kappa_{\psi}$ never vanishes, while $\kappa_{\psi \mathcal N}$ vanishes if and only if $\psi(\mathfrak p)=1$ or $\psi(\bar{\mathfrak p})=1$. In these cases, there exists a notion of {\it derived} cohomology class along the subvariety $\cC'$ of weight space, that we define as the Zariski closure of the points $\mathcal N \bar \lambda^t$ with $t \in \mathbb Z^{\geq 0}$, in an appropriate sense we make precise later on, \[ \kappa_{\psi,\infty}' \in H^1(K,\Lambda_K \otimes F_p(\psi^{-1})(\mathcal N)_{|\cC'}). \]
The specialization of this cohomology class at $\psi \mathcal N$ encodes the arithmetic information which is given by $\kappa_{\psi \mathcal N}$ in a non-exceptional situation. To be more explicit, denote by $u_{\mathfrak n}$ a fixed choice of an elliptic unit of conductor $\mathfrak n$, and to lighten notations, define
\begin{equation}
u_{\psi} = \prod_{\sigma \in \Gal(K_{\mathfrak n}/K)} (\sigma u_{\mathfrak n})^{\psi^{-1}(\sigma)} \in (\mathcal O_{K_{\mathfrak n}}^{\times} \otimes F)^{\psi},
\end{equation}
where $K_{\mathfrak n}$ is the field cut out by $\psi$. We expect a relation between $\kappa_{\psi \mathcal N}'$ and $u_{\psi}$, since they both lie in the same space (after applying a Kummer map), and this is the content of the main result of this note, which we now state.

Assume that $\psi(\mathfrak p)=1$, and let us define the following $\mathcal L$-invariant
\begin{equation}
\mathcal L(\psi) = (1-\psi^{-1}(\bar{\mathfrak p})) \cdot \log_p(\mathfrak p),
\end{equation}
where $\log_p$ stands for the usual $p$-adic Iwasawa logarithm.

Then, we have the following result, proved in Section \ref{11}.
\begin{theorem}\label{teo1}
Suppose that $\psi(\mathfrak p)=1$. Then,
\begin{equation}
\kappa_{\psi \mathcal N}' = \mathcal L(\psi) \cdot u_{\psi}.
\end{equation}
\end{theorem}

Although the previous result does not require any explicit mention to the theory of $p$-adic $L$-functions, it is fair to say that Katz's two variable $p$-adic $L$-function plays a prominent role in our results. The interplay between the Euler system of elliptic units and Katz's two variable $p$-adic $L$-function can be set as a very particular case of a wider theory. One may distinguish two main approaches to {\it construct} a $p$-adic $L$-function.
\begin{enumerate}
\item[(a)] Firstly, interpolating the algebraic parts of the special values of the classical $L$-function along the so-called {\it critical} region. This requires, as a starting point, the proof of certain algebraicity results.
\item[(b)] Secondly, as the image under a certain Perrin-Riou map of a family of cohomology classes, constructed along the so-called {\it geometric} region. These classes are typically obtained as the image under certain regulators of distinguished elements arising from the geometry of algebraic varieties.
\end{enumerate}
In both approaches, the $p$-adic $L$-function is completely characterized by the value at the points lying either at the {\it critical} or at the {\it geometric} region. Moreover, some Euler factors arise, measuring the discrepancy between the interpolation of $L$-values in the critical region and the interpolation of cohomology classes in the geometric region. The vanishing of these factors lead us to study exceptional zero phenomena. In the case of the Perrin-Riou map, the shape of these factors is $\frac{1-p^j \phi}{1-p^{-1-j}\phi^{-1}}$, where $j$ is related to the Hodge--Tate type of the character at which we are specializing, and $\phi$ refers to a Frobenius eigenvalue. As it is suggested for instance in \cite[Section 8]{KLZ} or \cite{LZ}, there are two kinds of Euler factors in the usual Perrin-Riou maps: those appearing in the numerator (which typically lead to an exceptional vanishing of the $p$-adic $L$-function via explicit reciprocity laws) and those appearing in the denominator (which lead to an exceptional vanishing of the cohomology class). While the former phenomenon has been widely studied, as far as we know the latter has only been discussed with the tools from Perrin-Riou theory in the setting of Heegner points in \cite{Cas} and for Beilinson--Flach elements in \cite{RiRo}. Nevertheless, similar results have been obtained by Bley \cite{Bl}, although there are some differences we later discuss.

Katz's two-variable $p$-adic $L$-function $L_{\mathfrak p}(K)(\cdot)$ is defined on the domain $\Hom_{\cont}(G_K, \bar{\mathbb Q}_p^{\times})$, but we can consider its restriction to
\begin{equation}\label{1-acl}
 \Sigma_{\psi} = \{ \psi \xi \mathcal N^s \lambda^t \text{ such that } (s,t) \in \mathbb Z_p^2 \},
\end{equation}
where $\xi$ is a finite order character of $p$-power conductor. This allows us to make use of the techniques relative to $p$-adic variation, sharing some points in common with Hida theory. We write $L_{\mathfrak p}(K,\psi)(\cdot)$ for the restriction of $L_{\mathfrak p}(K)(\cdot)$ to the subspace of characters given in \eqref{1-acl}, and denote $L_{\mathfrak p}(K,\psi)(\chi_{\triv}):=L_{\mathfrak p}(K)(\psi)$ and $L_{\mathfrak p}(K,\psi)(\mathcal N):=L_{\mathfrak p}(K)(\psi \mathcal N)$. In our case, and because of the dualities involved in the Perrin-Riou formalism, we are also interested in the function $L_{\mathfrak p}(K,(\psi')^{-1})$.

We can now describe the main ingredients involved in the proof of Theorem \ref{teo1}.

\begin{enumerate}
\item[(a)] An {\bf explicit reciprocity law} for Katz's two-variable $p$-adic $L$-function due to Yager. This expresses the special value $L_{\mathfrak p}(K,(\psi')^{-1})(\mathcal N)$ in terms of the image under a Perrin-Riou map of the cohomology class $\kappa_{\psi \mathcal N}$, and directly gives us that $\loc_{\mathfrak p}(\kappa_{\psi \mathcal N})=0$ when $\psi(\mathfrak p) = 1$, due to the vanishing of an Euler factor. Here, $\loc_{\mathfrak p}$ stands for the localization at $\mathfrak p$. The explicit description of the localization-at-$\mathfrak p$ map shows that we can conclude that $\kappa_{\psi \mathcal N}=0$ and consider the {\it derived} cohomology class. We refer the reader to Sections \ref{llei} and \ref{11} for a proper definition of {\it derived} class and for more details on that.

\item[(b)] A {\bf derived reciprocity law}, expressing the Bloch--Kato logarithm of the {\it derived} class in terms of $L_{\mathfrak p}(K,(\psi')^{-1})(\mathcal N)$. This requires an explicit description of the Perrin-Riou map, which at the norm character interpolates the Bloch--Kato logarithm and gives a map \[ \log_{\BK}:  H^1(K_{\mathfrak p},F_p(\psi^{-1})(\mathcal N)) \longrightarrow \mathbb D_{\dR}(F_p(\psi^{-1})) \simeq  F_p. \] Under the identification induced by the Kummer morphism, this map corresponds, in a sense that we later make precise, to the usual $p$-adic logarithm. Then, we have the following result, whose proof is given in Section \ref{11}.

\begin{propo}\label{main2}
Assume that $\psi(\mathfrak p)=1$. Then, \[ \log_p(\mathfrak p) \cdot L_{\mathfrak p}(K,(\psi')^{-1})(\mathcal N) = -(1-p^{-1}) \cdot \log_p(\loc_{\mathfrak p}(\kappa'_{\psi \mathcal N})). \]
\end{propo}

\item[(c)] The {\bf functional equation} for Katz's two-variable $p$-adic $L$-function (see \cite[p.90--91]{Gro}), which asserts that \[ L_{\mathfrak p}(K,(\psi')^{-1})(\mathcal N)  = L_{\mathfrak p}(K,\psi)(\chi_{\triv}). \]

\item[(d)] {\bf Katz's $p$-adic version of Kronecker limit formula}, expressing the special value $L_{\mathfrak p}(K,\psi)(\chi_{\triv})$ in terms of the elliptic unit $u_{\psi}$ \[ L_{\mathfrak p}(K,\psi)(\chi_{\triv}) = -(1-\psi^{-1}(\bar {\mathfrak p}))(1-\psi(\mathfrak p)p^{-1}) \cdot \log_p(u_{\psi}). \] In Section \ref{katz} we properly discuss the main features of Katz's two-variable $p$-adic $L$-function.

\end{enumerate}

As a by-product of the previous discussion, along the text we also deal with other instances of the exceptional zero phenomenon. The results of Section 4 encompass two main situations: the exceptional vanishing of $\kappa_{\psi \mathcal N}$; and the exceptional vanishing of the $p$-adic $L$-function $L_{\mathfrak p}(K,\psi)$, which is a more well-established phenomenon that has been widely studied in the literature and already appears in Katz's original work.

\vskip 12pt

Once these results have been developed, the last section of the article serves to analyze how our results fit with similar statements concerning exceptional zero phenomena. In particular, we emphasize the parallelism, but also the differences, with the theory of Heegner points, as well as the fact that these elliptic units may be seen as a particular case inside the theory of Beilinson--Flach elements, where different instances of the exceptional zero phenomena also appear. When $g$ is a theta series of an imaginary quadratic field where $p$ splits and we take the pair of modular forms $(g,g^*)$, \cite{RiRo} describes a connection between a {\it derived} Beilinson--Flach element, an elliptic unit and an special value of the Hida--Rankin $p$-adic $L$-function attached to $(g,g^*)$. The assumptions considered in loc.\,cit. (we had imposed that the Galois representation attached to $g$ was $p$-distinguished) excluded the possibility of elliptic units presenting an {\it exceptional zero}, so in a certain way the results of this article regarding exceptional zeros of elliptic units can be thought as a degenerate case inside the theory of Beilinson--Flach elements.
While our main theorem can be seen as the counterpart of \cite[Theorem B]{RiRo} in the framework of elliptic units, we point out that there is another exceptional zero phenomenon related to the vanishing of the numerator of the Perrin-Riou map, which in this case leads to a trivial zero of the Katz's two variable $p$-adic $L$-function (see Section 4.1) and which in the setting of Beilinson--Flach elements has been studied in \cite{LZ2}.

\vskip 12pt

{\bf Acknowledgements.}  It is a pleasure to sincerely thank Victor Rotger, who suggested me to work on this question and carefully read an earlier version of this manuscript. I also thank Kazim B\"uy\"ukboduk for useful comments and correspondence regarding his work, and Francesc Castell\`a for many enlightening conversations and helpful remarks during the writing of this note. Finally, it is a must to thank the anonymous referees for a very careful reading of the manuscript, whose comments notably contributed to improve the exposition and corrected several inaccuracies.

The author has been supported by Grant MTM2015-63829-P, as well as from the European Research Council (ERC) under the European Union's Horizon 2020 research and innovation programme (grant agreement No. 682152). The author has also received financial support through ``la Caixa" Fellowship Grant for Doctoral Studies (grant LCF/BQ/ES17/11600010).

\section{Circular units}\label{circulars}

Circular units constitute one of the first examples of Euler systems, and they play a key role in the proof of the classical Iwasawa main conjecture. We recall here some of their most relevant features because of the parallelism they keep with the theory of elliptic units. We discuss what the exceptional zero phenomenon represents in this case, and then we will compare this setting with that of elliptic units.

\subsection{Leopoldt's formula}

In this section, we denote by $\Lambda_{\cyc} := \mathbb Z_p[[\mathbb Z_p^{\times}]]$, and let $\mathcal W:=\Spf(\Lambda_{\cyc})$.
We fix a primitive, non-trivial even Dirichlet character of conductor $N$, \[ \chi: (\mathbb Z/N\mathbb Z)^{\times} \rightarrow F^{\times} \] where $F$ is a number field and $(p,N)=1$. We write $F_p$ for its completion at a prime lying above $p$. For our applications to exceptional zero phenomena, we are interested in the case in which $\chi(p)=1$.

The Kubota--Leopoldt $p$-adic $L$-function attached to $\chi$, $L_p(\chi,s)$, can be defined as the $p$-adic analytic function satisfying the interpolation property \[ L_p(\chi,n) = (1-\chi(p)p^{-n}) L(\chi,n), \quad \text{ for all } n \leq 0. \] Alternatively, we may understand it as a function defined over an appropriate rigid analytic space, sometimes called the {\it weight space}.

\begin{defi}
A {\it classical point} of $\mathcal W$ is a pair $(k,\xi)$, where $k$ is an integer and $\xi$ is a Dirichlet character of $p$-power conductor, corresponding to the homomorphism \[ z \mapsto z^{k-1} \xi(z). \]
\end{defi}

Then, the Kubota--Leopoldt $p$-adic $L$-function can be seen as an application \[ L_p(\chi, \cdot): \quad \mathcal W \longrightarrow \mathbb C_p \] defined in terms of an interpolation property for a subset of classical points. We warn the reader that there are several possible conventions regarding this function. Here, we closely follow the approach of \cite[Section 3]{PR}, and the $p$-adic $L$-function we have considered satisfies the interpolation property of Proposition 3.1.4 of loc.\,cit. See also \cite{BCDDPR} for a reformulation of those ideas in our language. Another standard way of presenting this $p$-adic $L$-function is discussed in \cite[Section 3]{Das}, and we will come back to this issue later on; there, the interpolation property involves appropriate twists by powers of the Teichm\"uller character, but both approaches are closely connected as shown in \cite[Section 3.1.5]{PR}: in particular, the $p$-adic $L$-values at integers $n$ with $n \equiv 1$ modulo $p-1$ agree.

Let $H$ denote the field cut out by $\chi$, and for a choice of a primitive $p^n$-th root of unity $\zeta_{p^n}$, let $H_n = H(\zeta_{p^n})$. Define the units
\begin{equation}
c_{\chi,n} := \prod_{a=1}^{N-1} (1-\zeta_{Np^n}^a)^{\chi^{-1}(a)} \in (\mathcal O_{H_n}^{\times} \otimes F)^{\chi},
\end{equation}
that behave under the norm maps as dictated by the theory of Euler systems: \[ \mathcal N_{H_n}^{H_{n+1}}(c_{\chi,n+1}) = \begin{cases} c_{\chi,n} & \text{ if } n \geq 1, \\  c_{\chi} \otimes (\chi(p)-1) & \text{ if } n=0. \end{cases}, \]
where $c_{\chi} = c_{\chi,0}$. As a word of caution, note that we have used the standard multiplicative notation, where the exponentiation $(1-\zeta_{Np^n}^a)^{\chi^{-1}(a)}$ means $(1-\zeta_{Np^n}^a) \otimes \chi^{-1}(a)$.

Hence, one can construct a norm compatible family of cohomology classes taking the image under the Kummer map $\delta$. More precisely, we consider
\begin{equation}\label{im-kummer}
\kappa_{\chi,n} := \delta (c_{\chi,n}) \in H^1(H_n, F_p(1))^{\chi} = H^1(H_n, F_p(\chi^{-1})(1)).
\end{equation}

These classes can be patched all together taking the projective limit for $n \geq 1$, resulting in an element $\kappa_{\chi,\infty}$
\begin{equation}\label{im-patching}
\kappa_{\chi,\infty} \in \lim_{\leftarrow} H^1(H_n, F_p(\chi^{-1})(1)) = H^1(\mathbb Q, \Lambda_{\cyc} \otimes F_p(\chi^{-1})(1)).
\end{equation}
As usual, $\Lambda_{\cyc}$ can be understood as a $p$-adic interpolation of the Tate twists, parameterized by pairs $(k,\xi)$, where $k \in \mathbb Z$ and $\xi$ is a Dirichlet character of $p$-power conductor. Hence, we are endowing $\Lambda_{\cyc}$ with the {\it tautological} action of $G_{\mathbb Q}$ (this is sometimes written in the literature as $\Lambda_{\cyc}(\underline{\varepsilon}_{\cyc})$ to emphasize the Galois action, where $\underline{\varepsilon}_{\cyc}$ is the so-called $\Lambda_{\cyc}$-adic cyclotomic character).
Let $F_{\xi,p}$ stand for the compositum of $F_p$ with the field of values of $\xi$. The specialization maps $\nu_{k,\xi}: \Lambda_{\cyc} \rightarrow F_{\xi,p}$ are ring homomorphisms sending the group-like element $a \in \mathbb Z_p^{\times}$ to $a^{k-1} (\chi \xi)^{-1}(a)$, and induce $G_{\mathbb Q}$-equivariant specialization maps \[ \nu_{k, \xi}: \Lambda_{\cyc} \rightarrow F_{\xi,p}(\xi^{-1})(k-1). \] This gives rise to a collection of global cohomology classes
\begin{equation}\label{im-classes}
\kappa_{k,\chi \xi} := \nu_{k,\xi}(\kappa_{\chi,\infty}) \in H^1(\mathbb Q, F_{\xi,p}(\chi \xi)^{-1}(k)).
\end{equation}

In order to state the following result, recall that the Gauss sum associated to a Dirichlet character $\eta$ of conductor $m$ and with values in a number field $L$ is defined by
\begin{equation}\label{gauss}
\mathfrak g(\eta) = \sum_{a=1}^{m-1} \zeta_m^a \otimes \eta(a) \in \mathcal O_{\mathbb Q(\zeta_m)}^{\times} \otimes F.
\end{equation}
From now on, $\exp_{\BK}^*$ stands for the Bloch--Kato dual exponential map and $\log_{\BK}$ for the Bloch--Kato logarithm. The following proposition is a reformulation of a classical result by Coleman \cite{Col}, using the formalism of Perrin-Riou regulators developed in \cite{PR}.

\begin{propo}[Coleman, Perrin-Riou]\label{perrin}
There exists a morphism of $\Lambda$-modules (referred to as the Coleman or the Perrin-Riou map) \[ \mathcal L_p: H^1(\mathbb Q_p, \Lambda_{\cyc} \otimes_{\mathbb Z_p} F_p(\chi^{-1})(1)) \longrightarrow I^{-1} \Lambda_{\cyc} \] satisfying that for all classical points $(k,\xi)$, the specialization map $\nu_{k,\chi \xi}(\mathcal L_p)$ is the homomorphism \[ \nu_{k,\chi \xi}(\mathcal L_p): H^1(\mathbb Q_p,F_{\xi,p}(\chi \xi)^{-1}(k)) \longrightarrow \mathbb D_{\dR}(F_{\xi,p}((\chi \xi)^{-1})(k)) \simeq F_{\xi,p} \] given by \[  \nu_{k,\chi \xi}(\mathcal L_p) = \frac{1}{\mathfrak g((\chi \xi)^{-1})} \cdot \frac{1-\chi \xi(p)p^{-k}}{1-(\chi \xi)^{-1}(p)p^{k-1}} \cdot \begin{cases} \frac{(-t)^k}{(k-1)!} \log_{\BK} & \text{ if } k \geq 1 \\ (-k)!t^k \exp_{\BK}^* & \text{ if } k<1, \end{cases} \]
where $t$ is Fontaine's $p$-adic analogue of $2\pi i$, and the target of both the Bloch--Kato logarithm and the dual exponential map is identified with $F_{\xi,p}$. Here, $I$ is the kernel of the specialization at $\nu_{1,\chi}$.
\end{propo}

We finally relate the image of the previously introduced class $\kappa_{\chi,\infty}$ under the Perrin-Riou regulator with the Kubota--Leopoldt $p$-adic $L$-function. See for instance \cite{PR} for a more detailed treatment of this result. Here, $\loc_p$ stands for the localization at $p$ of a global cohomology class.

\begin{theorem}\label{rec-kub}
Let $\chi$ stand for a non-trivial and even Dirichlet character. Then, the cohomology class $\kappa_{\chi,\infty} \in H^1(\mathbb Q, \Lambda_{\cyc} \otimes_{\mathbb Z_p} F_p(\chi^{-1})(1))$ introduced in \eqref{im-patching} satisfies \[ L_p(\chi, \cdot) = \mathcal L_p(\loc_p(\kappa_{\chi,\infty})). \]
\end{theorem}

The previous theorem can be seen as an equality in $\Lambda_{\cyc}$, and we may apply the specialization maps to both sides at any $\mathbb C_p$-valued point. From the previous results, and using Kummer's identifications again, it turns out that one has the equality
\begin{equation}
L_p(\chi \xi,1) = -\frac{(1-\chi \xi(p)p^{-1})}{(1-(\chi \xi)^{-1}(p))} \cdot \frac{\log_p(\loc_p(\kappa_{1,\chi \xi}))}{\mathfrak g((\chi \xi)^{-1})},
\end{equation}
whenever $(\chi \xi)(p) \neq 1$; if $(\chi \xi)(p)=1$, both the Euler factor in the denominator and the cohomology class in the numerator vanish. Recall that here we have identified $t \cdot \log_{\BK}$ with the Iwasawa $p$-adic logarithm.

Since $\chi$ is non-trivial, one has $L_p(\chi,1) \in \bar{\mathbb Q}_p^{\times}$. This suggests the existence of a {\it derived} cohomology class $\kappa_{1,\chi}'$ related with $L_p(\chi,1)$, which is the content of the following section. We recover this idea along the article, but anyway it is good to keep in mind that in this setting one also has a $p$-adic Kronecker's limit formula expressing the value of $L_p(\chi,1)$ in terms of a unit in the number field cut out by the character
\begin{equation}\label{circular}
L_p(\chi,1) = -\frac{(1-\chi(p)p^{-1})}{\mathfrak g(\chi^{-1})} \cdot \log_p \Big( \prod_{a=1}^{N-1} (1-\zeta^a)^{\chi^{-1}(a)} \Big).
\end{equation}
This result is generally due to Leopoldt (see also \cite{PR}), and is often called in the literature {\it Leopoldt's formula}.
The quantity $ \prod_{a=1}^{N-1} (1-\zeta^a)^{\chi^{-1}(a)}$ is typically referred as the {\it circular unit} attached to $\chi$ and we have denoted it by $c_{\chi}$.

As we have pointed out, we may instead consider a slightly different $p$-adic $L$-function, that we denote $L_{p,1}(\chi,n)$, and defined in terms of the interpolation property \[ L_{p,1}(\chi,n) = (1-\chi \omega^{n-1}(p)p^{-n}) \cdot L(\chi \omega^{n-1},n) \quad \text{ for all } n \leq 0, \] where $\omega$ stands for the modulo $p$ cyclotomic character (Teichm\"uller). A very interesting object of study in the theory of $\mathcal L$-invariants is $L_{p,1}'(\chi,0)$, in the case where $\chi \omega^{-1}(p)=1$ and therefore $L_{p,1}(\chi,0)=0$ due to an exceptional vanishing. Washington \cite{Was} provides a formula for the value of the derivative in terms of Morita's $p$-adic Gamma function, $\Gamma_p(x)$:
\begin{equation}\label{morita}
L_{p,1}'(\chi,0) = \log_p \Big( \prod_{a=1}^N \Gamma_p(a/N)^{\chi \omega^{-1}(a)} \Big) + L_{p,1}(\chi,0) \log_p(N).
\end{equation}
Hence, in the situation of {\it exceptional vanishing} $\chi \omega^{-1}(p)=1$, there is an exceptional zero for $L_{p,1}(\chi,s)$ at $s=0$ and one has that
\begin{equation}\label{circ-exc}
L_{p,1}'(\chi,0) = \log_p(v_{\chi}),
\end{equation}
where \[ v_{\chi} = \prod_{a=1}^N \Gamma_p(a/N)^{\chi \omega^{-1}(a)}. \]

In the case where one considers instead an {\it odd} Dirichlet character $\eta$ with $\eta(p)=1$, the determination of $L_{p,1}'(\eta \omega,0)$ is a particular case of Gross' conjectures, as studied first by Ferrero--Greenberg \cite{FG}, and then by Darmon--Dasgupta--Pollack (among others!) for arbitrary totally real fields. Here, this derivative is expressed in terms of the logarithm of a $p$-unit in the field cut out by the character.

\subsection{Exceptional zeros and circular units}

Suppose from now on that $\chi$ is a non-trivial, even, Dirichlet character of conductor $N$ with $\chi(p)=1$ and $(p,N)=1$. Then, the arguments of the previous section show that the specialization of the $\Lambda_{\cyc}$-adic class $\kappa_{\chi,\infty}$ at $\chi$ vanishes, that is, $\kappa_{1,\chi}=0$. Of course, this can be interpreted in terms of the vanishing of the denominator of the Perrin-Riou map. This section, where no claim of originality is made, explains how to obtain a formula for $L_p(\chi,1)$ involving $\kappa_{\chi,\infty}$, following for that the work of Solomon \cite{Sol} and B\"uy\"ukboduk \cite{Buy}, and also discusses how the vanishing of the numerator of the Perrin-Riou regulator at $s=0$ can be studied inside the framework developed in \cite{Ven}.

With the previous notations, let $T = F_p(\chi^{-1})(1)$ and $T^*= F_p(\chi)$, viewed as representations of $G_{\mathbb Q}$. We single out one of the $(p-1)$ connected components of $\mathcal W$, which corresponds to the choice of the residue class of 1 modulo $p-1$ and of the Iwasawa algebra $\Lambda = \mathbb Z_p[[1+p\mathbb Z_p]] \subset \mathbb Z_p[[\mathbb Z_p^{\times}]]$. After fixing a topological generator $\gamma$ of $1+p\mathbb Z_p$, one may consider the short exact sequence of $\mathbb Z_p$-modules \[ 0 \rightarrow \Lambda \otimes T \xrightarrow{\gamma-1} \Lambda \otimes T \rightarrow T \rightarrow 0 \] which induces a long exact sequence in cohomology. Since $H^0(\mathbb Q, T)=0$, \[ 0 \rightarrow H^1(\mathbb Q, \Lambda \otimes T) \xrightarrow{\gamma-1} H^1(\mathbb Q, \Lambda \otimes T) \xrightarrow{\mathcal N} H^1(\mathbb Q, T). \]

The image of $\kappa_{\chi,\infty}$ under the map $\mathcal N$ vanishes since $\chi(p)=1$, and hence there exists a unique \[ \kappa_{\chi,\infty}' \in H^1(\mathbb Q,\Lambda \otimes T) \] such that \[ \frac{\gamma-1}{\log_p (\gamma)} \cdot \kappa_{\chi,\infty}'  = \kappa_{\chi,\infty}. \]
The reason of normalizing by $\log_p(\gamma)$ is, as discussed in \cite[Section 3]{Buy}, that the {\it derived} class does not longer depend on the choice of the topological generator $\gamma$.

Summing all up, we have the following result. We refer the reader to \cite[Proposition 3.4]{Buy} for a more detailed discussion.
\begin{propo}\label{exc-circ}
If $\chi(p)=1$, the class $\kappa_{\chi,\infty}$ vanishes at the character $\chi$ and there exists a {\it derived} cohomology class $\kappa_{\chi,\infty}' \in H^1(\mathbb Q,\Lambda \otimes T)$ such that \[ \kappa_{\chi,\infty} = \frac{\gamma-1}{\log_p (\gamma)} \cdot \kappa_{\chi,\infty}'. \]
\end{propo}

\begin{remark}
It may be tempting to look for a relation between $\kappa_{1,\chi}'$ and the special value $L_p(\chi,1)$. However, the fact that the Euler factor $1-p^{k-1}$ is not analytic in the variable $k$ precludes the possibility of directly taking derivatives in the reciprocity law of Proposition \ref{perrin}. However, this can be remedied invoking Solomon's results, as we later see, connecting the order of the derived class with the special value at $s=1$. In the following sections we discuss how in a bigger weight space certain derivatives are related with the $p$-adic logarithm and others with the $p$-adic valuation.
\end{remark}

Let us provide a more explicit description of the previous result. In \cite{Buy} the author makes a connection between the value of $L_p(\chi,1)$ and Nekovar's pairings. In \cite[Corollary 2.11]{Buy} it is shown that \[ H_{\fin,p}^1(\mathbb Q,T) = (\mathcal O_H^{\times}[1/p])^{\chi} \otimes F_p, \]
is a two-dimensional space where we may explicitly construct a basis. Here, $\tilde H_{\fin}^1(\mathbb Q, T)$ stands for the Bloch--Kato Selmer group of classes which are unramified outside $p$ and de Rham at $p$. As before, we have written $H$ for the field cut out by $\chi$. The fact that this space is two-dimensional reflects the exceptional zero coming from the condition $\chi(p)=1$, which gives rise to an {\it extra} $p$-unit in the field cut out by the character.

Define \[ c_n = \mathcal N_{\mathbb Q(\zeta_{Nn})/H(\zeta_n)}(1-\zeta_{Nn}) \in (\mathcal O_{H(\zeta_n)}^{\times} \otimes F), \] and consider its $\chi$-part, $c_{\chi,n}$. The element $c_{\chi}:=c_{\chi,1}$ is called the {\it tame cyclotomic unit}, and agrees with the definition given in the previous section.
For a finite abelian extension $H'$ of $\mathbb Q$ of conductor $m$ we also define \[ \xi_{H'} = \mathcal N_{\mathbb Q(\zeta_{mp})/H'}(1-\zeta_{mp}). \] With a slight abuse of notation, we may identify the units with the cohomology classes obtained via the Kummer map. Then, it turns out that the collection \[ \xi = \xi_{\chi,\infty} := \{ e_{\chi} \xi_{H_n} \text{ for } n \geq 1 \} \in \lim_{\leftarrow} H^1(H_n, T), \] where for the sake of simplicity we have written $e_{\chi}$ for the $\chi$-projector, satisfies the Euler system distribution relation, and moreover $\xi_H = 1$. Proceeding as before, we obtain an element $z_{\chi, \infty}$ satisfying \[ \frac{\gamma-1}{\log_p(\gamma)} \times z_{\chi, \infty} = \xi. \] We call its bottom layer $z_{\chi} := z_{0,\chi} \in H^1(\mathbb Q, T)$ the {\it cyclotomic $p$-unit}, and $\{c_{\chi},z_{\chi}\}$ is a basis of $H_{\fin,p}^1(\mathbb Q,T)$. In \cite{Sol}, it is proved that $\log_p(c_{\chi}) = \ord_p(z_{\chi}) \in F_p$, where $\ord_p$ is the usual $p$-adic valuation. Of course, this depends on the choice of a prime of $F$ lying above $p$.

The interesting fact appears when the denominator of the Perrin-Riou map vanishes. To circumvent that problem, \cite[Section 6.1]{Buy} recasts to the principle of improved Perrin-Riou map, which allows to introduce a {\it primitive} $p$-adic $L$-function $\tilde L_p(\chi,s)$ vanishing at $s=1$. The main result of \cite{Buy} is the computation, via the theory of Nekovar's pairings, of a formula for $L_p(\chi,1)$, which asserts that
\begin{equation}
\tilde L_p'(\chi,1) = p \cdot L_p(\chi,1)= \frac{1-p}{\mathfrak g(\chi^{-1})} \times \log_p(c_{\chi}) = \frac{1-p}{\mathfrak g(\chi^{-1})} \times \ord_p(z_{\chi}).
\end{equation}

This also works for the case of an imaginary quadratic field when one only consider the $\mathbb Z_p$-extension which is ramified just over a fixed prime $\mathfrak p$ above $p$.

For the sake of completeness, we finish the section by analyzing what happens for $L_{p,1}(\chi,0)$, where we may follow the approach of \cite[Section 3]{Ven} to analyze the vanishing of the numerator in the Perrin-Riou map. To ease notations, let $\psi = \chi \omega^{-1}$ and write again $F$ for its field of values. In particular, we know that when $\psi(p)=1$, $L_{p,1}(\chi,0)=0$.


Since the numerator of the Perrin-Riou regulator vanishes, we can consider its derivative. Let $I$ stand for the augmentation ideal of $\Lambda$. Then, we define the {\it derivative} of the Perrin-Riou map $\mathcal L_p$ of Proposition \ref{perrin} as the application \[ \mathcal L_p': H^1(\mathbb Q_p, \Lambda \otimes_{\mathbb Z_p} F_p(\psi^{-1})(1)) \rightarrow I/I^2, \] i.\,e.\,, the composition of $\mathcal L_p$ with the projection $\{ \cdot \}: I \rightarrow I/I^2$.

Let $\kappa_{\psi}=(\kappa_{n,\psi})$ be the cohomology class we have previously introduced in \eqref{im-classes}. Following the same strategy as in \cite[Prop. 3.6]{Ven}, and identifying $\kappa_{0,\psi}$ with an element in $\Hom(\mathbb Q_p^{\times},\mathbb Q_p) \otimes F_p(\psi^{-1})$, one has that \[ \mathcal L_p'(\kappa_{\psi}) = -\mathfrak g(\psi^{-1})^{-1}(1-p^{-1})^{-1} \cdot \frac{\exp_{\BK}^*(\loc_p(\kappa_{0,\psi}))}{\log_p(\gamma)} \cdot \{\gamma\}, \] where we have identified $I/I^2$ with the multiplicative group $1+p\mathbb Z_p$. As in \cite[Section 5]{Ven}, we can relate the derivative of $\mathcal L_p$ with the derivative of $L_{p,1}(\chi,s)$ and obtain this way a formula for $L_{p,1}'(\chi,0)$ in terms of $\exp_{\BK}^*(\loc_p(\kappa_{0,\psi}))$.

\section{Elliptic units}

In this section we introduce Katz's two-variable $p$-adic $L$-function and the theory of elliptic units, following mainly \cite{deS} and \cite{Yag}, but adapting their results to the framework discussed before. We also recall the Perrin-Riou big logarithm and recast Yager's theorem, which gives an explicit reciprocity law analogue to Theorem \ref{rec-kub} in this setting. We recover the notations of the introduction, where $K$ is an imaginary quadratic field and we fix a prime $p$ which splits on $K$, i.e. $p\mathcal O_K = \mathfrak p \bar{\mathfrak p}$. We also fix an identification of $\mathbb C_p$ with $\mathbb C$ and embeddings of $\bar{\mathbb Q}$ to either of these fields, which are compatible with these identifications. Let $h$ denote the class number of $K$. Then, let $\pi_{\mathfrak p} \in \mathcal O_K$ be such that $\mathfrak p^h = \pi_{\mathfrak p} \mathcal O_K$, and define $\varpi_{\mathfrak p} = \pi_{\mathfrak p}/\pi_{\bar{\mathfrak p}}$. For simplicity, we assume that $\mathcal O_K^{\times} = \pm 1$ and that the discriminant of $K$ is an odd number $D<0$.

Consider also a non-trivial Hecke character of finite order $\psi$, of conductor $\mathfrak n$, where $(\mathfrak n,p)=1$. In the particular case that $\chi$ is a Dirichlet character of conductor $N:=\mathcal N_{K/\mathbb Q}(\mathfrak n)$, the Dirichlet character may be seen by restriction as an example of the kind of Hecke characters we are interested in, provided that $K$ is a quadratic field where all primes dividing $N$ split. As before, let $F$ stand for the field cut out by the character and $F_p$ for its completion.

\subsection{Elliptic units}\label{eu}
Elliptic units are the result of evaluating modular units at CM points. They give rise to units in abelian extensions of the imaginary quadratic field $K$ and are the counterpart of circular units for cyclotomic fields. They also constitute one of the key ingredients for the proof of the Iwasawa main conjecture for imaginary quadratic fields \cite{Rub}.

For the general construction of elliptic units, we refer the reader to the seminal work of Coates and Wiles \cite{CW}, or alternatively to Robert's original paper \cite{Rob}. Let us give an explicit description in the special setting where the conductor $\mathfrak n$ of $\psi$ satisfies that there exists a rational integer $N$ such that $\mathcal O_K/\mathfrak n \simeq \mathbb Z/N\mathbb Z$, and $\psi$ can be interpreted as a Dirichlet character of conductor $N$. We closely follow the survey \cite{BCDDPR} for that purpose.

Special values of $L$-series are encoded in terms of the so-called Siegel units $g_a \in \mathcal O_{Y_1(N)}^{\times} \otimes \mathbb Q$ attached to a fixed choice of primitive $N$-th root of unity $\zeta_N$ and a parameter $1 \leq a \leq N-1$. Its $q$-expansion is given by
\begin{equation}\label{siegel}
g_a(q) = q^{1/12} (1-\zeta_N^a) \prod_{n>0}(1-q^n \zeta_N^a)(1-q^n \zeta_N^{-a}).
\end{equation}

Let $\tau_{\mathfrak n} = \frac{b+\sqrt{D}}{2N}$, where $\mathfrak n = \mathbb Z N + \mathbb Z \frac{b+\sqrt{D}}{2}$. The classical and $p$-adic elliptic units are defined by
\begin{equation}\label{siegel2}
u_{a,\mathfrak n}:=g_a(\tau_{\mathfrak n}), \qquad u_{a, \mathfrak n}^{(p)}:=g_a^{(p)}(\tau_{\mathfrak p \mathfrak n}),
\end{equation}
being $g_a$ the infinite product of \eqref{siegel} and $g_a^{(p)}: =g_{pa}(q^p)g_a(q)^{-p}$.
As we did with circular units, we may define
\begin{equation}\label{def-unit}
u_{\psi} := \prod_{\sigma \in \Gal(K_{\mathfrak n}/K)} (\sigma u_{1,\mathfrak n})^{\psi^{-1}(\sigma)},
\end{equation}
where $K_{\mathfrak n}$ is the ray class field of $K$ of conductor $\mathfrak n$ and $u_{1,\mathfrak n} \in \mathcal O_{K_{\mathfrak n}}^{\times}$. In additive notation, this corresponds to $(\sigma u_{1,\mathfrak n}) \otimes \psi^{-1}(\sigma)$. This construction works in greater generality and one can always define the element $u_{\psi}$ (see \cite{Rob}). These units are the bottom layer of a norm compatible family of elliptic units over the two-variable $\mathbb Z_p$-extension $K_{\infty}$ of $K$.

Performing a similar construction to that of \eqref{im-kummer} and \eqref{im-patching}, the work of Katz \cite{Ka} and de Shalit \cite{deS} gives a cohomology class
\begin{equation}
\kappa_{\psi, \infty} \in H^1(K,\Lambda_K \otimes F_p(\psi^{-1})(\mathcal N)),
\end{equation}
where $\Lambda_K$ is the two-variable Iwasawa algebra of the introduction endowed with the tautological Galois action.
In particular, if $\eta$ is a Hecke character of infinity type $(\kappa_1,\kappa_2)$, the global class obtained by specializing $\kappa_{\psi,\infty}$ at $\eta$, although it arises from elliptic units, encodes information about a Galois representation of $K$ attached to a Hecke character.

\subsection{Katz's two-variable p-adic L-function of an imaginary quadratic field}\label{katz}
The classical two-variable $L$-function attached to $K$ and a Hecke $\psi$ is defined by \[ L(K, \psi, \kappa_1, \kappa_2) := \sideset{}{'}\sum_{\alpha \in \mathcal O_K} \psi(\alpha) \alpha^{-\kappa_1} \bar \alpha^{-\kappa_2}, \] where the sum is over the set of non-zero ideals of $\mathcal O_K$. This $L$-series allows us to recover the more familiar $L$-function attached to a character $\psi$ of an imaginary quadratic field, via the relation
\begin{equation}\label{1vs2}
L(K,\psi,s) = \frac{1}{2} L(K, \psi,s,s).
\end{equation}

We follow \cite[Section 3]{DLR1} for the construction of Katz's two variable $p$-adic $L$-function. Let $\mathfrak c \subset \mathcal O_K$ be an integral ideal of $K$, and let $\Sigma$ be the set of Hecke characters of $K$ of conductor dividing $\mathfrak c$. Define $\Sigma_K = \Sigma_K^{(1)} \cup \Sigma_K^{(2)} \subset \Sigma$ to be the disjoint union of the sets \[ \Sigma_K^{(1)} = \{ \psi \in \Sigma \text{ of infinity type } (\kappa_1, \kappa_2), \text{ with } \kappa_1 \leq 0, \kappa_2 \geq 1 \}, \] \[ \Sigma_K^{(2)} = \{ \psi \in \Sigma \text{ of infinity type } (\kappa_1, \kappa_2), \text{ with } \kappa_1 \geq 1, \kappa_2 \leq 0 \}. \] For all $\psi \in \Sigma_K$, the complex argument $s=0$ is a critical point for $L(\psi^{-1},s)$, and Katz's $p$-adic $L$-function is constructed interpolating the algebraic part of $L(\psi^{-1},0)$, as $\psi$ ranges over $\Sigma_K^{(2)}$.

Let $\hat \Sigma_K$ be the completion of $\Sigma_K^{(2)}$ with respect to the compact open topology on the space of $\mathcal O_{L_p}$-valued functions on a subset of $\mathbb A_K^{\times}$. By the work of Katz, there exists a $p$-adic analytic function \[ L_{\mathfrak p}(K): \hat \Sigma_K \longrightarrow \mathbb C_p, \] uniquely determined by the interpolation property that for all $\xi \in \Sigma_K^{(2)}$ of infinity type $(\kappa_1,\kappa_2)$,
\begin{equation}\label{katz3}
L_{\mathfrak p}(K)(\xi) = \mathfrak a(\xi) \times \mathfrak e(\xi) \times \mathfrak f(\xi) \times \frac{\Omega_p^{\kappa_1-\kappa_2}}{\Omega^{\kappa_1-\kappa_2}} \times L_{\mathfrak c}(\xi^{-1},0),
\end{equation}
where
\begin{enumerate}
\item $\mathfrak a(\xi) = (\kappa_1-1)! \pi^{-\kappa_2}$,
\item $\mathfrak e(\xi) = (1-\xi(\mathfrak p)p^{-1})(1-\xi^{-1}(\bar{\mathfrak p}))$,
\item $\mathfrak f(\xi) = D_K^{\kappa_2/2} 2^{-\kappa_2}$,
\item $\Omega_p \in \mathbb C_p^{\times}$ is a $p$-adic period attached to $K$,
\item $\Omega \in \mathbb C^{\times}$ is the complex period associated to $K$,
\item $L_{\mathfrak c}(\xi^{-1},s)$ is Hecke's $L$-function associated to $\xi^{-1}$ with the Euler factors at primes dividing $\mathfrak c$ removed.
\end{enumerate}

We have followed the conventions of \cite[Proposition 3.1]{BDP2}, which in turn follows from \cite[Section 5.3.0]{Ka}. Observe that the definition is not symmetric with respect to the primes $\mathfrak p$ and $\bar{\mathfrak p}$ above $p$, and hence we can also consider the function $L_{\bar{\mathfrak p}}(K)(\cdot)$.

The $p$-adic $L$-function $L_{\mathfrak p}(K)(\cdot)$ satisfies a functional equation
\begin{equation}\label{fe}
L_{\mathfrak p}(K)(\xi) = L_{\mathfrak p}(K)((\xi')^{-1} \mathcal N),
\end{equation}
where again $\xi'$ is the composition of $\xi$ with the complex conjugation (see \cite[pages 90--91]{Gro}). We remark that since our characters are unramified at $p$, the Gauss sum that sometimes appears in the interpolation formula is equal to 1. Finally, and according to this definition, the interpolation is over the special values of the form $L_{\mathfrak c}(\xi^{-1},0)$; this explains some {\it discrepancies} regarding certain conventions with the case of circular units.

It is possible to obtain an expression for the value of $L_{\mathfrak p}(K)(\psi)$ at finite order characters. This is usually referred to as the {\it $p$-adic Kronecker's limit formula}, and is due to Katz:
\begin{equation}\label{special}
L_{\mathfrak p}(K)(\psi) = \begin{cases} \frac{1}{2} \Big(\frac{1}{p}-1 \Big) \cdot \log_p(\pi_{\mathfrak p}^{1/h}) & \text{ if } \psi = 1;\\ -(1-\psi^{-1}(\bar{\mathfrak p}))(1-\psi(\mathfrak p)p^{-1}) \cdot \log_p(u_{\psi}) & \text { if } \psi \neq 1. \end{cases}
\end{equation}

Here, $h$ stands for the class number of $K$ and $\pi_{\mathfrak p}$ for a generator of the $\mathcal O_K$-ideal $\mathfrak p^h$.
Via the functional equation, we also have an expression for the value at the $\psi \mathcal N$
\begin{equation}\label{special2}
L_{\mathfrak p}(K)(\psi \mathcal N) = \begin{cases} \frac{1}{2} \Big(\frac{1}{p}-1 \Big) \cdot \log_p(\pi_{\mathfrak p}^{1/h}) & \text{ if } \psi = 1;\\ -(1-\psi(\mathfrak p))(1-\psi^{-1}(\bar{\mathfrak p})p^{-1}) \cdot \log_p(u_{(\psi')^{-1}}) & \text { if } \psi \neq 1. \end{cases}
\end{equation}


\begin{defi}
Let $L_{\mathfrak p}(K,\psi)(\cdot)$ stand for the restriction of the $p$-adic $L$-function of \eqref{katz3} to characters of the form $\psi \xi \lambda^{\kappa_1} (\lambda')^{\kappa_2}$, where $\xi$ is a character of $p$-power conductor and $\lambda$ is the character of infinity type $(1,0)$ presented in the introduction.

In particular, write $L_{\mathfrak p}(K,\psi)(\chi_{\triv}):=L_{\mathfrak p}(K)(\psi)$ and $L_{\mathfrak p}(K,\psi)(\mathcal N):=L_{\mathfrak p}(K)(\psi \mathcal N)$.

\end{defi}

\begin{remark}
Depending on the normalization we choose for the two variable $p$-adic $L$-function, the value $L_{\mathfrak p}(K)(\psi)$ may be affected by multiplication by a non-zero explicit rational number. Further, this number can depend on the conductor of $\psi$; however, since we are restricting the function to characters of the form $\psi \xi$, where $\xi$ has $p$-power conductor, we can adopt a suitable normalization such that our special value formulas always work.

Further, observe that the $p$-adic $L$-function of \cite[Theorem 6.3]{Buy} also differs from this one in the factor $(1-\xi^{-1}(\bar{\mathfrak p}))$.
\end{remark}

We finish this description of Katz's two-variable $p$-adic $L$-function by discussing its relation with the theory of improved $p$-adic $L$-functions. We say that a character is {\it analytic} if it is of the form $\psi \lambda^t$, with $t \in \mathbb Z^{\geq 0}$. The reason for this terminology is that they correspond to the subvariety of the weight space along which the Euler factors appearing in the Perrin-Riou map are analytic as functions in the variable $t$. Katz constructed in \cite[Section 7.2]{Ka} a one-variable $p$-adic $L$-function $L_{\mathfrak p}^*(K,\psi,k)$, such that the restriction of $L_{\mathfrak p}(K,\psi)$ to analytic characters yields the equation
\begin{equation}\label{1vs2p}
L_{\mathfrak p}(K,\psi)(\lambda^k) = (1-\psi^{-1}(\bar{\mathfrak p})\pi_{\bar{\mathfrak p}}^{-k/h}) L_{\mathfrak p}^*(K,\psi,k).
\end{equation}

The ratio of the two $p$-adic $L$-series is a $p$-adic analytic function of $k$, since $\pi_{\bar{\mathfrak p}}$ belongs to $\mathcal O_{K_{\mathfrak p}}^{\times}$. This ratio measures the difference between working with the ordinary $p$-stabilization of the Eisenstein series, $E_{k,\psi}^{(p)}$, and the $p$-depletion, $E_{k,\psi}^{[p]}$. Further, by a result of Katz \cite[Section 7.2]{Ka}, one has a relation between the $p$-adic $L$-function of a quadratic imaginary field and elliptic units, given by
\begin{equation}\label{rel1}
L_{\mathfrak p}^*(K,\psi,0) = -(1-\psi(\mathfrak p) p^{-1}) \cdot \log_p(u_{\psi}).
\end{equation}

We refer the reader to \cite[Section 4.3]{BD} for a more detailed exposition of this material.

\subsection{A reciprocity law for elliptic units}\label{llei}

In this section, we recall the existence of a Perrin-Riou map interpolating both the dual exponential map and the Bloch--Kato logarithm, as we did with circular units. We closely follow the treatment of \cite{deS}, recalling the main properties of this regulator map, whose source is the Iwasawa cohomology of the representation induced by a Hecke character, and which interpolates the Bloch--Kato logarithm and the dual exponential map, depending on the Hodge--Tate type of the character at which we specialize.

Although this is part of a rather general theory, we are interested in a more down-to-earth version of these results, which have been recovered by Loeffler and Zerbes in \cite{LZ} in the setting of {\it two-variable Perrin-Riou regulators}. In particular, Theorem 4.15 of loc.\,cit. gives an analogue to Proposition \ref{perrin} in the setting of elliptic units. We restrict to characters of $K$ of the form $\psi \lambda^{\kappa_1} (\lambda')^{\kappa_2}$, where $\lambda$ is the character of infinity type $(1,0)$ of the introduction and $\lambda'$ is its complex conjugate. Of course we may also consider twists by characters $\xi$ of $p$-power order, but we neglect this possibility so as to ease the exposition.

Recall that $\Lambda_K$ is the two-variable Iwasawa algebra attached to $K$ with the tautological Galois action. As with circular units, for any character $\eta=\lambda^{\kappa_1} (\lambda')^{\kappa_2}$ as above there exists a specialization map that we normalize, to ease the exposition, as the one inducing specializations of the form \[ \nu_{\kappa_1,\kappa_2}: \Lambda_K \otimes F_p(\psi^{-1})(\mathcal N) \rightarrow F_p(\psi^{-1})(\lambda^{\kappa_2} (\lambda')^{\kappa_1}). \] We identify again the target of both the dual exponential map and the Bloch--Kato logarithm with $F_p$. The following result has been established in \cite[Section 6.4, Appendix B]{LZ}, and here is presented in the language of \cite[Section 8]{KLZ}.

\begin{propo}\label{aquest-perrin}
There exists a morphism \[ \mathcal L_{\mathfrak p}: H^1(K_{\mathfrak p}, \Lambda_K \otimes F_p(\psi^{-1})(\mathcal N)) \longrightarrow \Lambda_K \] interpolating both the dual exponential map and the Bloch--Kato logarithm, and such that for any point $\eta$ of infinity type $(\kappa_1,\kappa_2)$ with $\xi=1$, the specialization of $\mathcal L$ at $\eta$ is the homomorphism \[ \nu_{\kappa_1,\kappa_2}(\mathcal L_{\mathfrak p}): H^1(K_{\mathfrak p}, F_p(\psi^{-1})(\lambda^{\kappa_2} (\lambda')^{\kappa_1})) \longrightarrow \mathbb D_{\dR}(F_p(\psi^{-1})(\lambda^{\kappa_2} (\lambda')^{\kappa_1})) \simeq  F_p \] given by
\[ \nu_{\kappa_1,\kappa_2}(\mathcal L_{\mathfrak p}) = \frac{1-\psi(\mathfrak p) \pi_{\mathfrak p}^{-\kappa_2/h}\pi_{\bar{\mathfrak p}}^{-\kappa_1/h}}{1-\frac{\psi^{-1}(\mathfrak p) \pi_{\mathfrak p}^{\kappa_2/h} \pi_{\bar{\mathfrak p}}^{\kappa_1/h}}{p}} \begin{cases} \frac{(-t)^{\kappa_2}}{\kappa_2!} \log_{\BK} & \text{ if } \kappa_2 > 0 \\ (-\kappa_2)!t^{\kappa_2} \exp_{\BK}^* & \text{ if } \kappa_2 \leq 0 . \end{cases} \]
\end{propo}

\begin{remark}
It is interesting to analyze the shape of the Euler factors and compare it with those of \cite[Theorem 3.5]{Cas}. To follow this parallelism, let $\kappa=\kappa_1+\kappa_2$ and $r=-\kappa_2$. Then,
\begin{equation}
\frac{1-\psi(\mathfrak p) \pi_{\mathfrak p}^{-\kappa_2/h}\pi_{\bar{\mathfrak p}}^{-\kappa_1/h}}{1-\frac{\psi^{-1}(\mathfrak p) \pi_{\mathfrak p}^{\kappa_2/h} \pi_{\bar{\mathfrak p}}^{\kappa_1/h}}{p}} = \frac{1-\psi(\mathfrak p) \pi_{\mathfrak p}^{r/h}\pi_{\bar{\mathfrak p}}^{(-\kappa-r)/h}}{1-\frac{\psi^{-1}(\mathfrak p) \pi_{\mathfrak p}^{-r/h} \pi_{\bar{\mathfrak p}}^{(\kappa+r)/h}}{p}} = \frac{1-{\psi(\mathfrak p)}\varpi_{\mathfrak p}^{r/h} \pi_{\bar{\mathfrak p}}^{-\kappa/h}}{1-\frac{\psi^{-1}(\mathfrak p)\varpi_{\mathfrak p}^{-r/h}\pi_{\bar{\mathfrak p}}^{\kappa/h}}{p}}.
\end{equation}

Our results concerning elliptic units can be seen as a counterpart of those for Heegner points when the cuspidal Hida family is replaced by an Eisentein series.
\end{remark}

\begin{remark}
As we will discuss in the last section, this also fits well with \cite[Proposition 3.2]{RiRo}, which is a reformulation of \cite[Theorem 8.1.7]{KLZ}; with the notations of loc.\,cit., if we fix $s=0$ and identify the modular forms $g$ and $h$ with two weight one theta series attached to the imaginary quadratic field $K$, we recover the map of Proposition \ref{aquest-perrin}. Further, observe that the numerology is coherent, and the condition $m>s$ defining the region of interpolation of the Bloch--Kato logarithm becomes in this case $\kappa_2>0$.
\end{remark}

The following result is the main theorem of \cite{Yag}.
\begin{propo}\label{lareclaw}
The cohomology class $\kappa_{\psi,\infty} \in H^1(K,\Lambda_K \otimes F_p(\psi^{-1}))$ satisfies \[ L_{\mathfrak p}(K,(\psi')^{-1}) = \mathcal L_{\mathfrak p}(\loc_{\mathfrak p}(\kappa_{\psi,\infty})). \]
\end{propo}

Again, the denominator of the Perrin-Riou regulator may vanish. Assume that $\psi(\mathfrak p)=1$. Then, we have the following:
\begin{enumerate}
\item [(i)] If $\kappa_1=\kappa_2=0$, then the numerator vanishes and the denominator equals $1-p^{-1}$.
\item [(ii)] If $\kappa_1=\kappa_2=1$, the numerator equals $1-p^{-1}$ and the denominator vanishes.
\end{enumerate}

\begin{remark}
For a fixed $\kappa_2$, both the numerator and the denominator are analytic functions on the variable $\kappa_1$.
\end{remark}

\section{Exceptional zeros and elliptic units}

We analyze different instances of exceptional zero phenomena and discuss the existence of {\it derived} cohomology classes and some of their properties. Our main result, stated as Theorem \ref{teo1} in the introduction, is about the exceptional vanishing of $\kappa_{\psi \mathcal N}$, but for the sake of convenience we also study the exceptional vanishing of Katz's two variable $p$-adic $L$-functon at $\psi$ in Section \ref{00}. Then, in Section \ref{11} we discuss the different cases of exceptional zeros at $\psi \mathcal N$ and prove the main result of the note. Finally, Section \ref{caso} discusses the special case where $\psi$ is the restriction of a Dirichlet character, suggesting a tantalizing connection with the theory of circular units.

\subsection{Specialization at the character $\psi$}\label{00}

We assume that the condition $\psi(\bar{\mathfrak p}) = 1$ is satisfied. We begin this section by discussing the vanishing of Katz's two-variable $p$-adic $L$-function at the character $\psi$ under this hypothesis. Indeed, from \eqref{1vs2p}, it is straightforward that $\psi(\bar{\mathfrak p})=1$ is a necessary and sufficient condition for the vanishing of $L_{\mathfrak p}(K,\psi)(\chi_{\triv})$ when the classical value is non-zero.

Until otherwise stated, and so as to shorten the notation, derivatives of $p$-adic $L$-functions are considered along the character $\lambda$ (that is, along the $H_{\infty}$ direction).
\begin{propo}\label{form-ell-unit}
Assume that $\psi(\bar{\mathfrak p})=1$. Then, $L_{\mathfrak p}(K,\psi)(\chi_{\triv})=0$, and
\begin{equation}\label{katz-buscada}
L_{\mathfrak p}'(K,\psi)(\chi_{\triv}) = \log_p(\pi_{\mathfrak p}^{1/h}) \cdot (1-\psi(\mathfrak p) p^{-1}) \cdot \log_p(u_{\psi}).
\end{equation}
\end{propo}
\begin{proof}
This follows by taking derivatives in \eqref{1vs2p} and applying the special value formula of \eqref{rel1}. Observe that we have also used that $\log_p(\pi_{\mathfrak p}^{1/h}) = -\log_p(\pi_{\bar{\mathfrak p}}^{1/h})$.
\end{proof}

The Euler factors arising in the Perrin-Riou map are analytic in the variable $\kappa_1$ once the value of $\kappa_2$ is fixed. In particular, we may consider the Perrin-Riou map with $\kappa_2=0$ fixed. Writing $\kappa_{\kappa_1,\psi}$ for the corresponding specialization of the $\Lambda$-adic class, one has \[ (1-\pi_{\bar{\mathfrak p}}^{\kappa_1/h}/p) \cdot L_{\mathfrak p}(K,\psi)(\chi_{\triv}) = (1-\pi_{\bar{\mathfrak p}}^{-\kappa_1/h}) \cdot \exp_{\BK}^*(\loc_{\mathfrak p}(\kappa_{\kappa_1,(\psi')^{-1}})), \] for all $\kappa_1 \geq 0$. Taking derivatives with respect to $\kappa_1$ at both sides and evaluating at $\kappa_1=0$, we get that
\begin{equation}\label{eva1}
(1-p^{-1}) L_{\mathfrak p}'(K,\psi)(\chi_{\triv}) = -\log_p(\pi_{\mathfrak p}^{1/h}) \cdot \exp_{\BK}^*(\loc_{\mathfrak p}(\kappa_{(\psi')^{-1}})),
\end{equation}
where $\kappa_{(\psi')^{-1}}$ stands for the specialization at the trivial character.

Here, by class field theory, \[ H^1(K_{\mathfrak p},F_p(\psi^{-1})) \simeq \Hom_{\cont}(K_{\mathfrak p}^{\times},K_{\mathfrak p}) \otimes F_p(\psi^{-1}), \] and $\loc_{\mathfrak p}(\kappa_{(\psi')^{-1}})$ corresponds to the evaluation at the inverse of a local uniformizer, $\pi_{\mathfrak p}^{-1/h}$. In particular, there is a non-canonical isomorphism with $F_p$. Under this isomorphism, the dual exponential map corresponds to the $\ord$ map. Then, we may identify an element of the space \[ \Hom_{\cont}(K_{\mathfrak p}^{\times},K_{\mathfrak p}) \otimes F_p(\psi^{-1}) \] with its image under evaluation at $\pi_{\mathfrak p}^{-1/h}$.

Therefore, combining Proposition \ref{form-ell-unit} with equation \eqref{eva1}, we get the following.

\begin{propo}
Assume that $\psi(\bar{\mathfrak p})=1$. Then, with the previous identifications,
\begin{equation}
\loc_{\mathfrak p}(\kappa_{(\psi')^{-1}}) = -(1-p^{-1})(1-\psi(\mathfrak p)p^{-1}) \times \log_p(u_{\psi}).
\end{equation}
\end{propo}

\subsection{Specialization at the character $\psi \mathcal N$}\label{11}

The motivation for this section comes from the following lemma.
\begin{lemma}\label{dos-fact}
The following relation between the unit $u_{\psi}$ and the specialization of $\kappa_{\psi,\infty}$ at the character $\psi \mathcal N$ holds:
\begin{equation}\label{dous-factores}
\log_p(\kappa_{\psi \mathcal N}) = (1-\psi^{-1}(\mathfrak p)) \cdot (1-\psi^{-1}(\bar{\mathfrak p})) \cdot \log_p(u_{\psi}).
\end{equation}
In particular, $\log_p(\kappa_{\psi \mathcal N})$ vanishes if and only if $\psi(\mathfrak p)=1$ or $\psi(\bar{\mathfrak p})=1$.
\end{lemma}

\begin{proof}
This follows after combining the $p$-adic Kronecker limit formula given in \eqref{special2}, now applied to $(\psi')^{-1}$, with Proposition \ref{lareclaw}. Observe that the factor $1-\psi(\mathfrak p)/p$ cancels out, and the same occurs for the two minus signs (the one coming from the Perrin-Riou map and that of the $p$-adic Kronecker limit formula).
\end{proof}

Further, since both $\kappa_{\psi \mathcal N}$ and $u_{\psi}$ are $\psi$-units, the $p$-adic logarithm defines an isomorphism and we may upgrade the previous lemma to an equality in $(\mathcal O_{K_{\mathfrak n}}^{\times} \otimes F)^{\psi}$. Observe that we are implicitly using that the localization map corresponds to the injection of global units inside local units, and this allows us to conclude that the global class $\kappa_{\psi \mathcal N}$ also vanishes when $\psi(\mathfrak p)=1$ or $\psi(\bar{\mathfrak p})=1$.

\vskip 12pt

We distinguish now three different situations.
\begin{enumerate}
\item [(a)] If $\psi(\mathfrak p)=1$, the cohomology class $\kappa_{\psi \mathcal N}=0$. One can construct a {\it derived} class whose derivative along the direction $\lambda$ is computed and expressed in terms of the unit $u_{\psi}$. If $\psi(\bar{\mathfrak p}) \neq 1$, the special value $L_{\mathfrak p}(K,\psi)(\mathcal N)$ does not vanish and it is related with the Bloch--Kato logarithm of the {\it derived} class.
\item [(b)] If $\psi(\bar{\mathfrak p})=1$, both $\kappa_{\psi \mathcal N}=0$ and $L_{\mathfrak p}(K,\psi)(\mathcal N)=0$. The logarithm of the {\it derived} local class $\loc_{\mathfrak p}(\kappa_{\psi \mathcal N}')$, once the value of $\kappa_2$ is fixed, can be expressed in terms of $L_{\mathfrak p}'(K,\psi)(\mathcal N)$.
\item [(c)] When $\psi(\mathfrak p)=\psi(\bar{\mathfrak p})=1$, both $\kappa_{\psi \mathcal N}=0$ and $L_{\mathfrak p}(K,\psi)(\mathcal N)=0$. Observe that this can be understood as an extension of the subcases (a) and (b). Then, $\kappa_{\psi \mathcal N}'=0$ too, and there is a notion of {\it second derivative} of the cohomology class, whose logarithm is related with $L_{\mathfrak p}'(K,\psi)(\mathcal N)$.
\end{enumerate}

{\bf Case (a).} Suppose that $\psi(\mathfrak p)=1$. Let $\mathcal C$ be the line of weight space obtained by taking the Zariski closure of all the points of the form $\mathcal N \lambda^t$ (that is, fixing $\kappa_2=1$). We are considering the $\mathbb C_p$-points as a rigid analytic space, and realizing all these characters as elements of this space. Analogously, define $\mathcal C'$ as the closure of the points of the form $\mathcal N \bar{\lambda}^t$. At the level of $p$-adic $L$-functions, and under the map of Proposition \ref{lareclaw}, the point $\mathcal N \lambda^t$ in $\mathcal C$ corresponds to $L_{\mathfrak p}(K,(\psi')^{-1})(\mathcal N \bar \lambda^t)$, or alternatively and via the functional equation, to $L_{\mathfrak p}(K,\psi)(\lambda^t)$.


In this situation of {\it exceptional zeros}, the Euler factor in the denominator of the Perrin-Riou map of Proposition \ref{aquest-perrin} vanishes, and hence one must carry out some new constructions to obtain a formula relating the special value of the $p$-adic $L$-function with an appropriate cohomology class. For that purpose, we can argue the existence of a {\it derived} cohomology class arising from elliptic units, which is directly related with the special value of the derivative of Katz's two variable $p$-adic $L$-function. Observe that derivatives of cohomology classes must be taken along $\mathcal C'$, since elsewhere the Euler factors are not analytic (and not even continuous!).

\begin{propo}\label{der-ell}
It holds that $\kappa_{\psi \mathcal N} = 0$, and there exists a derived cohomology class along $\mathcal C'$  \[ \kappa_{\gamma,\psi,\infty}' \in H^1(K, \Lambda_K \otimes F_p(\psi^{-1})(\mathcal N)_{|\mathcal C'}) \] satisfying that \[ \kappa_{\psi|\mathcal C'} = (\gamma-1) \kappa'_{\gamma, \psi,\infty}, \] where $\gamma$ is a fixed topological generator of $1+p\mathbb Z_p$.
\end{propo}
\begin{proof}
The vanishing of the local class $\loc_{\mathfrak p}(\kappa_{\psi \mathcal N})$ directly follows from Lemma \ref{dos-fact} and the discussion after it.

The construction of the derived class follows the same argument than in the case of circular units explained in Section 2, and which is already present in the work of Bley \cite{Bl}. After fixing a topological generator $\gamma$ of $1+p\mathbb Z_p$, one may consider the short exact sequence of $\mathbb Z_p$-modules \[ 0 \rightarrow \Lambda_K \otimes F_p(\psi^{-1})(\mathcal N)_{|\mathcal C'} \xrightarrow{\gamma-1} \Lambda_K \otimes F_p(\psi^{-1})(\mathcal N)_{|\mathcal C'} \rightarrow F_p(\psi^{-1})(\mathcal N) \rightarrow 0 \] which induces a long exact sequence in cohomology. Since $H^0(K, F_p(\psi^{-1})(\mathcal N))=0$, \[ 0 \rightarrow H^1(K, \Lambda_K \otimes F_p(\psi^{-1})(\mathcal N)_{|\mathcal C'}) \xrightarrow{\gamma-1} H^1(K, \Lambda_K \otimes F_p(\psi^{-1})(\mathcal N)_{|\mathcal C'}) \xrightarrow{\mathcal N} H^1(K, F_p(\psi^{-1})(\mathcal N)). \]

As we have just seen, the image of $\kappa_{\psi,\infty}$ under the map $\mathcal N$ vanishes, and hence there exists a unique \[ \kappa_{\gamma,\psi,\infty}' \in H^1(K,\Lambda_K \otimes F_p(\psi^{-1})(\mathcal N)_{|\mathcal C'}) \] such that \[ \frac{\gamma-1}{\log_p (\gamma)} \times \kappa_{\gamma,\psi,\infty}'  = \kappa_{\gamma,\psi,\infty}. \]
\end{proof}

\begin{remark}
If we normalize dividing by $\log_p(\gamma)$, the specialization of the resulting class at the character $\psi \mathcal N$ does not depend on $\gamma$ (see \cite[Section 3]{Buy}). We define $\kappa_{\psi,\infty}':=\frac{\kappa_{\gamma,\psi,\infty}'}{\log_p(\gamma)}$.
\end{remark}

We can relate the cohomology class with $L_{\mathfrak p}(K)(\cdot)$ at $\psi \mathcal N$.
\begin{defi}
Let $\mathcal E_{\mathfrak p}(K,(\psi')^{-1})(\cdot)$ stand for the function defined over the set of characters of the form $\lambda^{\kappa_1} \bar \lambda$ (and then extended to their Zariski closure over weight space), and given by
\begin{equation}\label{1exp}
\mathcal E_{\mathfrak p}(K,(\psi')^{-1})(\lambda^{\kappa_1}\bar \lambda) = \Big( 1-\frac{\pi_{\mathfrak p}^{1/h} \pi_{\bar{\mathfrak p}}^{\kappa_1/h}}{p} \Big) L_{\mathfrak p}(K,(\psi')^{-1})(\lambda^{\kappa_1} \bar \lambda).
\end{equation}
\end{defi}

\begin{lemma}
The function $\mathcal E_{\mathfrak p}(K,(\psi')^{-1})$ satisfies that for a character $\eta$ of infinity type $(\kappa_1,1)$
\begin{equation}\label{2exp}
\mathcal E_{\mathfrak p}(K)(\psi \eta) =  -\Big(1-\pi_{\mathfrak p}^{-1/h} \pi_{\bar{\mathfrak p}}^{-\kappa_1/h} \Big) \cdot t \log_{\BK}(\loc_{\mathfrak p}(\kappa_{\psi \eta})).
\end{equation}
In particular, the function $\mathcal E_{\mathfrak p}(K,(\psi')^{-1})$ is zero at the norm character.
\end{lemma}
\begin{proof}
The first part follows from Proposition \ref{lareclaw}. The second statement is due to the vanishing of the specialization of the cohomology class at $\psi \mathcal N$ (recall that for characters of infinity type $(\kappa_1,1)$ the Bloch--Kato logarithm interpolates the usual logarithm map).
\end{proof}

This is sufficient to prove Theorem \ref{teo1}. Recall again that we may identify the Bloch--Kato logarithm with the usual $p$-adic logarithm under Kummer's isomorphism.

\begin{remark}
For the following result, we need to use that $(\log_{\BK}(\kappa))'=\log_{\BK}(\kappa')$. This can be easily seen by considering the natural isomorphism between the weight space and $\mathbb Z_p[[X]]$. Then, the class $\kappa$ corresponds to a function $f$ vanishing at $0$, and hence there is another function $g$ such that $f=X \cdot g$. The Bloch--Kato logarithm is a linear morphism between a $\mathbb Z_p[[X]]$-module and a field embedded in $\mathbb C_p$, that we may denote with the letter $\Phi$. Then, \[ \Phi(f)'|_{X=0} = \Phi(g)|_{X=0}, \] as desired.
\end{remark}

\begin{theorem}
It holds that \[ \log_p(\pi_{\mathfrak p}^{1/h}) \cdot L_{\mathfrak p}(K,(\psi')^{-1})(\mathcal N) = -(1-p^{-1}) \cdot \log_p(\loc_{\mathfrak p}(\kappa'_{\psi \mathcal N})). \] Moreover, \[\kappa_{\psi \mathcal N}' = \log_p(\pi_{\mathfrak p}^{1/h}) \cdot (1-\psi^{-1}(\bar{\mathfrak p})) \cdot u_{\psi}, \] where $\loc_{\mathfrak p}$ stands here for the composition of the Kummer map with localization at $\mathfrak p$.
\end{theorem}
\begin{proof}
The first part follows by considering the derivative of the function $\mathcal E_{\mathfrak p}(K,(\psi')^{-1})$ using both \eqref{1exp} and \eqref{2exp}. That is, the left hand side is the result of deriving \eqref{1exp} and evaluating at $\kappa_1=1$, where the Euler factor vanishes; similarly, the right hand side is the result of deriving \eqref{2exp} and evaluating again at $\kappa_1=1$, where the class vanishes. Observe that we have identified $t \log_{\BK}$ with the $p$-adic logarithm, as usual.

The second part directly follows after comparing the first statement with the special value formula given in \eqref{special2}. Further, since we are taking a derivative along the $\bar{\mathfrak p}$-ramified extension $H_{\infty}'/K$, the resulting {\it derived} class could only have non-zero valuation at the prime $\bar{\mathfrak p}$. This cannot be the case when $\psi(\bar{\mathfrak p}) \neq 1$ since there are no extra $\bar{\mathfrak p}$-units, and the $p$-adic logarithm is therefore an isomorphism for the $\psi$-component of the unit group. If $\psi(\mathfrak p)=\psi(\bar{\mathfrak p})=1$ the same conclusion follows by recasting to the arguments of \cite{Bl}. We discuss this situation in more detail in the following sections.
\end{proof}

This result slightly differs from the work of Bley \cite{Bl}. There, the author takes the derivative along the direction $\mathcal C$, which corresponds to the $\mathbb Z_p$-extension which is ramified at the prime $\mathfrak p$. In that case, the derivative of the Perrin-Riou map agrees with the {\it order} (and not with the logarithm), which is coherent with the fact that when $\psi(\mathfrak p)=1$ there is a $\mathfrak p$-unit in the $\psi$-component. Hence, our results are coherent with the computations of \cite[Section 6.4]{Buy} relating the valuation of the extra $\mathfrak p$-unit with the logarithm of the elliptic unit $u_{\psi}$.

\vskip 12pt

{\bf Case (b).} Assume now that $\psi(\bar{\mathfrak p})=1$. Observe that via the functional equation for Katz's two-variable $L$-function, this leads to \[ L_{\mathfrak p}(K,(\psi')^{-1})(\mathcal N) = 0, \] and $L_{\mathfrak p}'(K,(\psi')^{-1})(\mathcal N)$ can be related with the {\it derived} cohomology class constructed in the previous section. We recall that the derivative of the $p$-adic $L$-function always means derivative along the $H_{\infty}$ direction. As we did in Proposition \ref{der-ell}, we may take the class \[ \kappa_{\psi,\infty}' \in H^1(K, \Lambda_K \otimes F_p(\psi^{-1})(\mathcal N)_{|\mathcal C'}). \]

\begin{propo}\label{2b}
Assume that $\psi(\bar{\mathfrak p})=1$. Then,
\begin{equation}
(1-\psi^{-1}(\mathfrak p)) \cdot L_{\mathfrak p}'(K,(\psi')^{-1})(\mathcal N) = -(1-\psi(\mathfrak p)p^{-1}) \cdot  \log_p(\loc_{\mathfrak p}(\kappa_{\psi \mathcal N}')).
\end{equation}
\end{propo}
\begin{proof}
This directly follows by considering the derivative with respect to the variable $\kappa_1$ in the reciprocity law of Proposition \ref{lareclaw} when $\kappa_2$ is fixed.
\end{proof}

The derivative $L_{\mathfrak p}'(K,\psi)(\mathcal N)$ is related, via the funcional equation, with the derivative at the trivial character along the direction $\lambda'$, and a priori we do not know any expression for that value in terms of $u_{\psi}$, which would allow us to prove an analogue of Theorem \ref{teo1} in this setting. Further, in this case the derived class $\kappa_{\psi \mathcal N}'$ is no longer a unit, but a $\bar{\mathfrak p}$-unit, and the $p$-adic logarithm is not sufficient to characterize $\kappa_{\psi \mathcal N}'$ (one needs to use the information about its $p$-adic valuation). It would be nice to understand these results in the framework provided by \cite[Section 4]{BD}, and we hope to come back to this issue in forthcoming work.

In the framework of circular units, Gross' factorization formula \cite{Gro}, combined with the results of the previous section, allows us to express, after considering the identifications provided by Kummer's isomorphisms, the class $\kappa_{\psi \mathcal N}'$ as a linear combination of a circular unit and an elliptic unit. We discuss this in Section \ref{caso}

\begin{remark}
Following Proposition \ref{2b}, we can extend our computations for the derivative of the cohomology class to arbitrary directions. Indeed, we may consider the function $L_{\bar{\mathfrak p}}(K)(\cdot)$, thus complementing our picture.
\begin{itemize}
\item When $\psi(\mathfrak p)=1$, $L_{\bar{\mathfrak p}}(K)(\cdot)$ vanishes at $\psi \mathcal N$ and in this case we can relate {\it the derivative of $\loc_{\mathfrak p}(\kappa_{\psi \mathcal N})$ along the direction $\mathcal C$} with {\it the derivative of $L_{\bar {\mathfrak p}}(K)$ at $\psi$ along the $(-1,0)$ direction}.
\item When $\psi(\bar{\mathfrak p})=1$, {\it the derivative of $\loc_{\mathfrak p}(\kappa_{\psi \mathcal N})$ along the direction $\mathcal C$} can be explicitly expressed in terms of $L_{\bar{\mathfrak p}}(\psi \mathcal N)$.
\end{itemize}
\end{remark}

\vskip 12pt

{\bf Case (c).} When $\psi(\bar{\mathfrak p}) = \psi(\mathfrak p)=1$, we expect that the cohomology class $\kappa_{\psi \mathcal N}'$ vanishes since $L_{\mathfrak p}(K,\psi)(\mathcal N)=0$. More precisely, and due to the symmetry between both directions $\lambda$ and $\lambda'$, we may combine the results of previous sections with \cite[Theorem 6.13]{Buy} to conclude that the class $\kappa_{\psi \mathcal N}'$ is zero.

When this happens, we may take a {\it second} derivative of the cohomology class \[ \kappa_{\gamma,\psi,\infty}'' \in H^1(K, \Lambda_K \otimes F_p(\psi^{-1})(\mathcal N)_{|\mathcal C'}). \] Now, we normalize the class dividing by $\log_p(\gamma)^2$ in such a way that its value at the character $\psi \mathcal N$, $\kappa_{\psi \mathcal N}''$, does not depend on $\gamma$. Proceeding in the same way as before, we may obtain the following result.

\begin{propo}\label{2.c}
Assume that $\psi(\bar{\mathfrak p})=\psi(\mathfrak p)=1$. Then,
\begin{equation}
\log_p(\pi_{\mathfrak p}^{1/h}) \cdot L_{\mathfrak p}'(K,(\psi')^{-1})(\mathcal N) = -2(1-p^{-1}) \cdot \log_p(\loc_{\mathfrak p}(\kappa_{\psi \mathcal N}'')).
\end{equation}
\end{propo}

\begin{proof}
Consider again the reciprocity law of Proposition \ref{aquest-perrin} and fix again $\kappa_2=1$. When evaluating at the norm character, it happens that both the Euler factor of the denominator and the special value $L_{\mathfrak p}(K,(\psi')^{-1})(\mathcal N)$ are zero. Hence, multiplying both sides by $1-\pi_{\mathfrak p}^{-1/h} \pi_{\bar{\mathfrak p}}^{-\kappa_1/h}$ and taking twice the derivative with respect to $\kappa_1$, we obtain the result.
\end{proof}

\begin{remark}
We point out, just for the sake of completeness, that another interesting instance of the {\it exceptional zero phenomenon} can be observed in \cite[Proposition 3.5]{BDP2}, which asserts that a {\it self-dual} character of infinity type $(1+j,-j)$ with $j \geq 0$ satisfies that the evaluation of Katz's two-variable $p$-adic $L$-function at $\nu$ agrees, up to multiplication by some periods and gamma factors, with a classical $L$-value times $(1-\nu^{-1}(\bar{\mathfrak p}))^2$. Again, if $\nu(\bar{\mathfrak p})=0$ we observe the presence of an {\it exceptional zero}.

In the special case where we consider characters of infinity type $(1,0)$, Agboola \cite{Agb} studies a variant of the $p$-adic BSD conjecture for CM elliptic curves concerning special values of Katz's two-variable $p$-adic $L$-function. Here, we are again in a situation where our same condition leads to an exceptional vanishing.
\end{remark}

\subsection{Interactions with the theory of circular units}\label{caso}

Observe that in the case where $\psi(\mathfrak p)=1$ we have described the {\it derived} cohomology class $\kappa_{\psi \mathcal N}'$ as an explicit multiple of the elliptic unit $u_{\psi}$. However, when $\psi(\bar{\mathfrak p})=1$ this is no longer possible, since we cannot express in terms of $u_{\psi}$ the derivative of the Katz's two variable $p$-adic $L$-function at $\psi$ along the $\lambda'$-direction.

In general, one may wish to determine the derivatives of the $p$-adic $L$-function along the different directions of the weight space. We know the derivative along the $\lambda$-direction, and it turns out that in some particular cases we can further determine the derivative along the {\it norm} direction.

This is the case when $\psi$ is a finite order Hecke character which comes from the restriction to $G_K$ of a Dirichlet character (that is, a $G_{\mathbb Q}$ character), and hence one can invoke Gross' factorization formula, which is the main result of \cite{Gro}. To begin with, consider that $\psi$ comes from a Dirichlet character and that $\psi(\bar{\mathfrak p})=1$. In this case, we have determined the derivative of $L_{\mathfrak p}(K,\psi)$ at $\chi_{\triv}$ along the direction $\lambda$, and this is \[ \frac{\partial L_{\mathfrak p}(K,\psi)}{\partial \lambda} = \log_p(\pi_{\mathfrak p}^{1/h}) \cdot (1-\psi(\mathfrak p)p^{-1}) \times \log_p(u_{\psi}). \] Similarly, using Gross' factorization with the conventions about Gauss sums followed in the article, we have that \[ \frac{\partial L_{\mathfrak p}(K,\psi)}{\partial \mathcal N} = -(1-\psi(\mathfrak p)p^{-1}) \cdot L_{p,1}'(\psi^{-1} \chi_K \omega, 0) \times \log_p(c_{\psi}), \] where $\chi_K$ stands for the quadratic character attached to $K$ and $\omega$ is the Teichm\"uller character. Observe that here $L_p(\psi^{-1} \chi_K \omega,0)=0$ due to the running assumptions.

Then, one can determine the derivative along any direction: for $\eta=\psi \lambda^a \mathcal N^b$, the derivative of $L_{\mathfrak p}(K,\psi)$ at the trivial character along the direction $\eta$ is given by
\begin{equation}\label{any}
a \cdot \log_p(\pi_{\mathfrak p}^{1/h}) \cdot(1-\psi(\mathfrak p) p^{-1}) \cdot \log_p(u_{\psi}) - b \cdot  L_{p,1}'(\psi^{-1} \chi_K \omega,0) \cdot (1-\psi(\mathfrak p)p^{-1}) \cdot \log_p(c_{\psi}).
\end{equation}

\begin{remark}
There is only one direction along which this value is zero; as discussed in the inspiring presentation \cite{Gr}, this has significant applications towards Iwasawa theory.
\end{remark}

Now, the functional equation gives the derivative of $L_{\mathfrak p}(K,(\psi')^{-1})$ at the norm character along any direction, and in particular, for the direction $\lambda$, we have to set $a=1$ and $b=-1$ (this is the direction $-(\lambda')$ at the trivial character). Then, Proposition \ref{2b} yields that, up to an element in the kernel of the logarithm, \[ \kappa_{\psi \mathcal N}' = \alpha u_{\psi} + \beta c_{\psi} \] where $c_{\psi}$ is the circular unit introduced in Section \ref{circulars}, and $\alpha$ and $\beta$ can be determined combining \eqref{any} with \eqref{circular}: \[ \alpha = -\log_p(\pi_{\mathfrak p}^{1/h}) (1-\psi^{-1}(\mathfrak p)), \] \[ \beta = -L_{p,1}'(\psi^{-1} \chi_K \omega,0)(1-\psi^{-1}(\mathfrak p)). \]

If we further assume that $\psi(\mathfrak p)=1$, Proposition \ref{2.c} gives a new expression for the second derivative of the cohomology classm again up to an element in the kernel of the logarithm, \[ \kappa_{\psi \mathcal N}'' = \alpha' u_{\psi} + \beta' c_{\phi}, \] where now \[ \alpha' = -\log_p(\pi_{\mathfrak p}^{1/h})^2/2, \] \[ \beta' = -\log_p(\pi_{\mathfrak p}^{1/h}) L_{p,1}'(\psi^{-1} \chi_K \omega,0)/2. \]

\section{The exceptional zero phenomenon: Beilinson--Flach elements and elliptic units}\label{bf}

In this last section, we emphasize the interplay between the results we have presented until now and other related works in this direction. In particular, we study how the phenomena we have discussed arise in some of the other Euler systems presented in the survey \cite{BCDDPR}, focusing on two main aspects:
\begin{enumerate}
\item[(a)] Elliptic units can be seen as the natural substitute of Heegner points when instead of considering a cusp form, one takes an Eisenstein series.
\item[(b)] Elliptic units can be recast in terms of Beilinson--Flach elements, when we take two weight one modular forms corresponding to theta series of the same imaginary quadratic field where the prime $p$ splits.
\end{enumerate}

As it has been extensively discussed in the literature, there is a striking parallelism between the theory of Heegner points and that of elliptic units. Following this analogy, this note may be read as the counterpart of \cite{Cas} when the cusp form $f$ is replaced by an Eisenstein series. With the notations introduced in loc.\,cit., where the weight space is modeled by a weight variable (that we denote with the letter $k$) and an anticyclotomic variable (denoted with the letter $t$), an exceptional zero arises at the point $(k,t)=(2,0)$ when ${\bf a}_p(\hf)=1$ (the associated elliptic curve has split multiplicative reduction at $p$) and we specialize at the character $\psi \mathcal N$. There is a clear interplay between that setting and ours, but we would like to point out some of the differences:
\begin{itemize}
\item In \cite{Cas} the author extends the $p$-adic Gross--Zagier formula of \cite{BDP} and finds an explicit expression for its value at the norm character, which is different from zero. However, in our setting it may occur that the $p$-adic $L$-function vanishes both at $\psi \mathcal N$ and at $\psi$. This simultaneous vanishing of the Euler factor and the $p$-adic $L$-function gives rise to a higher order vanishing of the {\it derived} class $\kappa_{\psi,\infty}$ at the character $\psi \mathcal N$. This can be understood via \cite[Eq. 0.2]{Cas}, where the specialization of the higher dimensional Heegner cycle (whose role is now played by the cohomology class coming from the elliptic unit) and the {\it Heegner class} (whose role is now played by the unit $u_{\psi}$) are related by the factor \[ \Big( 1-\frac{p^{k/2-1}}{\nu_k({\bf a}_p)} \Big)^2. \]
    In our case, however, the link is via the factor \[ (1-\psi^{-1}(\bar{\mathfrak p})\mathfrak p^{(\kappa_1-1)/h} \bar{\mathfrak p}^{(\kappa_2-1)/h}) \cdot (1-\psi^{-1}(\mathfrak p)\mathfrak p^{(\kappa_2-1)/h} \bar{\mathfrak p}^{(\kappa_1-1)/h}), \] and hence there are two possible (and independent) sources of vanishing.

\item In our setting there are two points where the exceptional zero phenomenon emerges: the character $\psi$ and the character $\psi \mathcal N$. In \cite{Cas} the vanishing of the numerator in the Perrin-Riou map would occur at $(k,t)=(0,-1)$, where there is not a clear {\it geometric meaning} of this phenomenon.
\end{itemize}

In any case, the similitude between his main result and ours is evident, expressing a {\it derived} cohomology class as a certain $\mathcal L$-invariant times a more classical avatar. Obviously, his $\mathcal L$-invariant encodes information both about the elliptic curve and the imaginary quadratic field $K$.

\subsection{Elliptic units and Beilinson--Flach elements}

Elliptic units can be understood as a degenerate case of the theory of Beilinson--Flach elements. To make this statement more precise, let \[ g = \sum_{n \geq 1} a_nq^n \in S_1(N_g, \chi_g), \quad h = \sum_{n \geq 1} b_nq^n \in S_1(N_h,\chi_h) \] be two normalized newforms, and let $V_g$ and $V_h$ denote the Artin representations attached to them. Consider also $V_{gh}:=V_g \otimes V_h$, and let $H$ be the smallest number field cut out by this representation. We fix a prime number $p$ which does not divide $N_gN_h$. Label the roots of the $p$-th Hecke polynomial of $g$ and $h$ as \[ X^2-a_p(g)X+\chi_g(p)=(X-\alpha_g)(X-\beta_g) \quad X^2-a_p(h)X+\chi_h(p)=(X-\alpha_h)(X-\beta_h). \] Let \[ g_{\alpha}(q) = g(q) - \beta_g q(q^p), \qquad h_{\alpha}(q) = h(q)-\beta_h h(q^p) \] denote the $p$-stabilization of $g$ (resp. $h$) on which the Hecke operator $U_p$ acts with eigenvalue $\alpha_g$ (resp. $\alpha_h$). Let $F$ be a number field containing both the Fourier coefficients of $g$ and $h$ and the eigenvalues for the $p$-th Hecke polynomials. We can attach in a natural way two {\it canonical} differentials $\omega_{g_{\alpha}}$ and $\eta_{g_{\alpha}}$ to the weight one modular form $g$, as it is recalled in \cite[Sections 2,3]{RiRo}. The reinterpretation of the main results of \cite{KLZ} in \cite{RiRo} and \cite{RiRo2} establishes the existence of cohomology classes \[ \kappa(g_{\alpha},h_{\alpha}), \kappa(g_{\alpha},h_{\beta}), \kappa(g_{\beta},h_{\alpha}), \kappa(g_{\beta},h_{\beta}) \in H^1(\mathbb Q, V_{gh} \otimes F_p(1)), \] and also gives an explicit reciprocity law which in slightly rough terms asserts that \[ \Big( 1-\frac{1}{p\alpha_g \beta_h} \Big) \cdot  \log^{-+}(\kappa_p(g_{\alpha},h_{\alpha})) = (1-\alpha_g \beta_h) \cdot L_p(g,h,1) \pmod{F^{\times}}. \] Here, $L_p(g,h,s)$ stands for the Hida--Rankin $p$-adic $L$-function attached to the convolution $g \otimes h$, $\kappa_p(g_{\alpha},h_{\alpha})$ is the restriction of the cohomology class to a decomposition group at $p$, and $\log^{-+}$ is the result of applying the Bloch--Kato logarithm to a certain projection of the local class followed by the pairing with the canonical differentials. The proof of this result is based on considering Hida families $\hg$, $\hh$ interpolating $g_{\alpha}$ and $h_{\alpha}$, and on proving the corresponding equality over a dense set of points of the weight space. We refer the reader to \cite{RiRo} for a complete discussion of the results.

Now, let $g^*$ stand for the twist of $g$ by the inverse of its nebentype. When $h_{\alpha}=g_{1/\beta}^*$, the Euler factor $1-\alpha_g \beta_h$ is zero, and Proposition 3.12 of loc.\,cit.\,establishes that both $\kappa(g_{\alpha},g_{1/\beta}^*)$ and $\kappa(g_{\beta},g_{1/\alpha}^*)$ vanish and moreover the authors prove the existence of a {\it derived} cohomology class $\kappa'(g_{\alpha},g_{1/\beta}^*) \in H^1(\mathbb Q, V_{gg^*} \otimes F_p(1))$ satisfying
\begin{equation}\label{citar}
\log^{-+}(\kappa_p'(g_{\alpha},g_{1/\beta}^*)) = \mathcal L(\ad^0(g_{\alpha})) \cdot L_p(g,g^*,1) \pmod{F^{\times}},
\end{equation}
being $\mathcal L(\ad^0(g_{\alpha}))$ the $\mathcal L$-invariant of the adjoint of the weight one modular form $g_{\alpha}$.

Other interesting cases arise when both $g$ and $h$ are theta series attached to the same quadratic imaginary field where the prime $p$ splits. In \cite[Section 6]{RiRo}, the authors prove a formula establishing an explicit connection between the Hida--Rankin $p$-adic $L$-function attached to the pair of modular forms $(g,g^*)$ and Katz's two variable $p$-adic $L$-function. Indeed, let $g=\theta(\psi)$, the theta series attached to the character $\psi$. Then, \cite[Theorem 6.2]{RiRo} asserts that for any $s \in \mathbb Z_p$ the following equality holds up to multiplication by $F^{\times}$:
\begin{equation}\label{fact-bf}
L_p(g,g^*,s) = \frac{1}{\log_p(u_{\psi_{\ad}})} \cdot \zeta_p(s) \cdot L_p(\chi_K \omega,s) \cdot L_{\mathfrak p}(K,\psi_{\ad})(\mathcal N^s),
\end{equation} being $\psi_{\ad}=\psi/\psi'$. Note that $\psi_{\mathrm{ad}}$ is a ring class character, regardless of whether $\psi$ is so or not. Then, according to \cite{RiRo},
\begin{equation}
L_p(g,g^*,0) = \log_p(v_1) \pmod{F^{\times}},
\end{equation}
where $v_1$ is the norm of a generator $v$ of the one-dimensional space $(\mathcal O_H^{\times}[1/p] \otimes F)^{G_{\mathbb Q}}$.

Observe that according to \cite[Section 5.3]{RiRo}, this gives a link between Beilinson--Flach elements and the cohomology classes coming from elliptic units via the Kummer map, expressed as
\begin{equation}
\kappa'(g_{\alpha},g_{1/\beta}^*) = \log_p(v_1) \cdot v \pmod{F^{\times}}.
\end{equation}
Additionally, the derived Beilinson--Flach element is also related with the cohomology class $\kappa_{\psi \mathcal N}$ via the factorization formula \eqref{fact-bf} and the results of the preceding sections.

As sketched in \cite[Section 5.2]{RiRo2}, the factorization formula \eqref{fact-bf} admits a counterpart in the case where $g$ and $h$ are no longer self-dual. In this case,
\begin{equation}
L_p(g,h,0) = \log^{-+}(\kappa_p(g_{\alpha},h_{\alpha})) = \frac{\log_p(u_{\psi_1}) \cdot \log_p(u_{\psi_2})}{\log_p(u_{g_{\alpha}})} \pmod{F^{\times}},
\end{equation}
where $\psi_1 = \psi_g \psi_h$ and $\psi_2 = \psi_g \psi_h'$, $u_{\psi_i}$ is the elliptic unit attached to $\psi_i$ and $u_{g_{\alpha}}$ is the Stark unit attached to the adjoint representation of $g_{\alpha}$.

Then, and following \cite{RiRo2},
\begin{equation}
\kappa(g_{\alpha},h_{\alpha}) = \mathfrak C \cdot u_2,
\end{equation}
with $\mathfrak C$ an explicit constant involving $u_{\psi_1}$, $u_{g_{\alpha}}$ and certain periods explicitly described in loc.\,cit., and $u_2=u_{\psi_2}u_{\psi_2'}$, where as usual $\psi_2'$ is the composition of $\psi_2$ with the complex conjugation.

An interesting observation is that the case where the Euler system of elliptic units presents an exceptional zero never arises in the setting of \cite{RiRo}, due to the regularity assumptions which are assumed in loc.\,cit (the fact of $g$ being $p$-distinguished). Hence, our results may be seen as a degenerate case of the theory of Beilinson--Flach elements for weight one modular forms.

Let us be more precise in this last sentence. Let $\hg$ stand for the Hida family of CM theta series whose weight $\kappa_1$ specialization has characteristic Hecke polynomial at $p$ given by \[ (x-\mathfrak p^{(\kappa_1-1)/h})(x-\bar{\mathfrak p}^{(\kappa_1-1)/h}). \] Similarly, let $\hh$ be the canonical Hida family of CM forms, such that its weight $\kappa_2$ specialization has characteristic Hecke polynomial at $p$ \[ (x-\psi(\mathfrak p)\mathfrak p^{(\kappa_2-1)/h})(x-\psi(\bar{\mathfrak p})\bar{\mathfrak p}^{(\kappa_2-1)/h}). \] Then, as we had already anticipated, Proposition \ref{aquest-perrin} may be seen as a degenerate case of \cite[Proposition 3.2]{RiRo}, where we do not consider the twist by the cyclotomic character (we fix $s=0$). Observe that there, the role played by Katz's two-variable $p$-adic $L$-function is done not exactly by the Hida--Rankin $p$-adic $L$-function, but by its product with the $c$-factor
\begin{equation}\label{factor-c}
c^2(1-\chi_g(c)^{-1}\chi_h(c)^{-1}),
\end{equation}
where $c$ is a fixed integer number relatively prime to $6pN_gN_h$.

\begin{remark}
The connection between Beilinson--Flach elements and units (in this case circular units) is also exploited in \cite{Das}, where the proof of the main result, a factorization formula for the Rankin--Selberg $p$-adic $L$-function, rests on a explicit comparison between a certain unit constructed via the theory of Beilinson--Flach elements and a circular unit. However, the approach used in loc.\,cit. is quite different, since the unit is constructed via the specialization of the Beilinson--Flach class at a point of weight $(2,2,1)$.
\end{remark}

\subsection{Beilinson--Flach elements and exceptional zeros}

As we have pointed out, elliptic units may be understood as a special case inside the theory of Beilinson--Flach elements, where the two modular forms are theta series attached to the same imaginary quadratic field. Hence, it is reasonable to expect that the two phenomena we have described concerning exceptional zeros also arise in this setting. This section serves to recall the main characteristics of the exceptional zero phenomenon for Beilinson--Flach elements, following closely the discussions of \cite{RiRo} and \cite{LZ2}.

With the notations of the previous section, let $\hg$ and $\hh=\hg^*$ Hida families interpolating two self-dual modular forms $g_{\alpha}$ and $h_{\alpha}=g_{1/\beta}^*$, respectively. We assume that $\hg \in \Lambda_{\hg}[[q]]$, where $\Lambda_{\hg}$ is a finite flat extensions of the Iwasawa algebra $\Lambda_{\cyc}=\mathbb Z_p[[\mathbb Z_p^{\times}]]$. Write $\mathcal W = \Spf(\Lambda_{\cyc})$ and $\mathcal W_{\hg} = \Spf(\Lambda_{\hg})$. Let $y_0$ be a weight one point of $\Lambda_{\hg}$ such that $\hg_{y_0}=g_{\alpha}$ and $\hg^*_{y_0}=g_{1/\beta}^*$.

The work of \cite{KLZ} attaches to $(\hg,\hg^*)$ a three-variable family of cohomology classes $\kappa(\hg,\hg^*)$ parameterized by points $(y,z,s) \in \mathcal W_{\hg} \times \mathcal W_{\hg} \times \mathcal W$ of weights $(\ell,m,s)$. More precisely, if $\mathbb V_{\hg}$ and $\mathbb V_{\hg^*}$ stand for Hida's $\Lambda$-adic Galois representations afforded by $\hg$ and $\hg^*$, respectively, and $\underline{\varepsilon}_{\cyc}$ is the $\Lambda$-adic cyclotomic character, \[ \kappa(\hg,\hg^*) \in H^1(\mathbb Q, \mathbb V_{\hg} \hat \otimes \mathbb V_{\hg^*} \hat \otimes \Lambda_{\cyc}(\underline{\varepsilon}_{\cyc}^{-1})(1)). \] Observe that here, and as a matter of convention, we have used the inverse of the tautological action over $\Lambda_{\cyc}$, and this is why we have written $\Lambda_{\cyc}(\underline{\varepsilon}_{\cyc}^{-1})$ (we may avoid this by invoking the appropriate functional equation).

As it follows from the discussion of \cite[Sections 8,10]{KLZ}, the three-variable Hida--Rankin $p$-adic $L$-function $L_p(\hg,\hg^*)$ is the image of the class $\kappa(\hg,\hg^*)$ under a Perrin-Riou map. Again, the numerator or the denominator may vanish in some {\it exceptional} cases.

For the precise statements concerning Beilinson--Flach classes, we refer the reader to the notations of \cite{RiRo}. As a first observation, we have that in this self-dual case, and according to \cite[Theorem 9.4]{Das} and the computations of \cite{RiRo}, one has that $L_p(g,g^*,0)=L_p(g,g^*,1)$. The main results we want to discuss here are the following ones:
\begin{enumerate}
\item[(i)] When we specialize both $\hg$ and $\hg^*$ at a fixed weight one point $y_0$, and the cyclotomic variable $s$ is set as $s=0$, the denominator of the Perrin-Riou map is zero and the cohomology class $\kappa(\hg,\hg^*)(y_0,y_0,0)$ vanishes. Then, the explicit reciprocity law of \cite{KLZ} is substituted by a {\it derived} reciprocity law relating the {\it derived} cohomology class with $L_p(g,g^*,0)$, up to multiplication by an $\mathcal L$-invariant. This is precisely equation \eqref{citar}.
\item[(ii)] When we specialize both $\hg$ and $\hg^*$ at weight one, and the cyclotomic variable $s$ is set as $s=1$ the numerator of the Perrin-Riou map is zero {\bf but} $L_p(g,g^*,1)$ does not vanish (at least generically). This is because \cite[Thm. B]{KLZ} contains the correction factor of \eqref{factor-c} at the $L$-function side, which vanishes in this case.
\end{enumerate}

\vskip 12pt

{\bf Exceptional vanishing of the denominator of the Perrin-Riou big logarithm.} The denominator of the Perrin-Riou regulator introduced in \cite[Prop.3.2]{RiRo} vanishes at all points $(y,y,\ell-1)$ of weight $(\ell,\ell,\ell-1)$. In particular, specializing $\hg$ and $\hg^*$ at the weight one modular forms $g$ and $g^*$ respectively, it turns out that \[ \log^{-+}(\kappa_p'(\hg,\hg^*)(y_0,y_0,0)) = \mathcal L(\ad^0(g_{\alpha})) \cdot L_p(g,g^*,0) = \mathcal L(\ad^0(g_{\alpha}))  \cdot L_p(g,g^*,1) \pmod{F^{\times}}. \] To shorten our notations, we have written $\kappa_p'(\hg,\hg^*)$ for the localization at $p$ of the class.

We would like to emphasize the similitude with our main results for elliptic units, which also relate the logarithm of the {\it derived} cohomology class with a special value of Katz's two-variable $p$-adic $L$-function, up to multiplication by a certain $\mathcal L$-invariant.

It turns out that the special value $L_p(\hg,\hg^*)(y,y,0)$ is related via the functional equation with $L_p(\hg,\hg^*)(y,y,1)$ and here, to determine its value, we can follow the same approach than in this note: over the line corresponding to those points of weight $(\ell,\ell,\ell)$, the $p$-adic $L$-function factors due to the analyticity of an Euler factor (this is properly developed in \cite{Hi-ad}), and this allows us to obtain an explicit expression of the special value $L_p(g,g^*,1)$ via Galois deformation techniques. This expression involves units and $p$-units in the field cut out by the Galois representation $V_{gg^*}$.

\vskip 12pt

{\bf Exceptional vanishing of the numerator of the Perrin-Riou big logarithm.} The results we present now closely follow \cite{LZ2} and are the counterpart of those developed in \cite{RiRo}. We include it here for the sake of completeness, and to illustrate how this exceptional phenomenon arises in a setting which is germane to ours.

Consider specializations of $\kappa(\hg,\hg^*)$ at weights $(y,z,m)$, where $\wt(z)=m$. Then, if $\alpha_{g_y}$ and $\beta_{g_y}$ stand for the eigenvalues of the $p$-th Hecke polynomial of $g_y$, with $\ord(\alpha_{g_y}) \leq \ord(\beta_{g_y})$, the Euler factor in the numerator of the Perrin-Riou map is \[ 1-\frac{\alpha_{g_z}}{\alpha_{g_y}}. \] This factor is zero for all points of weight $(\ell,\ell,\ell)$. But this does not mean that the $p$-adic $L$-function vanishes at those points, since the explicit reciprocity law of \cite[Thm. B]{KLZ} contains the factor \[ c^2-c^{2s+2-\ell-m} = c^2(1-c^{m-\ell}) \] multiplying the value of $L_p(\hg,\hg^*)$, where $c$ is a fixed positive integer coprime with both $6p$ and the level of $g$. At the points where $\wt(z)=m$, the Perrin-Riou map interpolates the Bloch--Kato dual exponential map, and
\begin{equation}\nonumber
c^2  (1-c^{m-\ell})  \Big(1-\frac{\alpha_{g_y}}{p\alpha_{g_z}} \Big)   L_p(\hg,\hg^*)(y,z,m) = \Big(1-\frac{\alpha_{g_z}}{\alpha_{g_y}} \Big) \cdot \exp_{\BK}^{*-+}(\kappa_p(\hg,\hg^*)(y,z,s)),
\end{equation}
where $\exp_{\BK}^{* -+}$ stands for the composition of the projection to a certain subspace of $V_{gg^*} \otimes F_p(1)$ followed by the dual exponential map and the pairing with the canonical differentials.

Since both sides of the previous equation vanish along the line $y=z$, $\wt(z)=m$, we may consider the derivative at a point $(y,y,\ell)$, obtaining the expression
\begin{equation}\nonumber
c^2  (1-p^{-1}) \dot \log_p(c)  L_p(\hg,\hg^*)(y,y,\ell) = \Big(\frac{-\alpha_{\hg_y}'}{\alpha_{\hg_y}} \Big) \cdot \exp_{\BK}^{* -+}(\kappa_p(\hg,\hg^*)(y,y,\ell)),
\end{equation}
up to multiplication by $F^{\times}$. Here, $\alpha_{\hg}'$ stands for the derivative of the Iwasawa function $\alpha_{\hg}$ along the weight direction.

Additionally, invoking Hida's result on the existence of an improved $p$-adic $L$-function \cite{Hi-ad}, we get that, whenever $\wt(y)=1$, \[ \log_p(c) = \exp_{\BK}^{* -+}(\kappa_p(\hg,\hg^*)(y,y,\ell)) \pmod{F^{\times}}. \]

Observe now that \[ L_p(\hg,\hg^*)(y,y,\ell) = L_p(\hg,\hg^*)(y,y,\ell-1) \pmod{F^{\times}}, \] and hence \[ L_p(\hg,\hg^*)(y,y,\ell-1) = \mathcal L(\ad^0(g_{\alpha}))^{-1} \cdot \log^{-+}(\kappa_p'(\hg,\hg^*)(y,y,\ell-1))  \pmod{F^{\times}}. \]
Consequently, up to multiplication by $F^{\times}$, we have the equality
\begin{equation}\label{prob}
\frac{\log_p(c)}{\mathcal L(\ad^0(g_{\alpha}))^2} \cdot \log^{-+}(\kappa'_p(\hg,\hg^*)(y,y,\ell-1)) =  \exp_{\BK}^{* -+}(\kappa_p(\hg,\hg^*)(y,y,\ell))
\end{equation}

In particular, modulo $F^{\times}$, one has
\begin{equation}
\frac{\log_p(c)}{\mathcal L(\ad^0(g_{\alpha}))^2} \cdot \log^{-+}(\kappa'_p(\hg,\hg^*)(y_0,y_0,0)) =  \exp_{\BK}^{* -+}(\kappa_p(\hg,\hg^*)(y_0,y_0,1)).
\end{equation}

Observe that this is coherent with the computations of \cite{RiRo}, from where it follows that, up to multiplication by a scalar in $F^{\times}$, \[ \log^{-+}(\kappa_p'(\hg,\hg^*)(y,y,\ell-1)) = \mathcal L(\ad^0(g_{\alpha}))^2. \] This is a consequence of Hida's improved factorization and the {\it derived} version of the explicit reciprocity law of \cite{KLZ}.

\end{document}